\documentclass{amsart}
\usepackage{amsmath,amssymb,amsthm,geometry,graphics,tikz,mdwlist}
\usepackage[all]{xy}
\usetikzlibrary{arrows.meta}

\DeclareMathOperator{\Ker}{Ker}

\DeclareMathOperator*{\colim}{colim}

\DeclareMathOperator{\Hom}{Hom}

\DeclareMathOperator{\RHom}{\mathrm{R}\!\Hom}
\DeclareMathOperator{\cHom}{\mathcal{H}\hspace{-0.3pt}om}
\DeclareMathOperator{\End}{End}

\DeclareMathOperator{\cEnd}{\mathcal{E}\hspace{-0.3pt}nd}

\DeclareMathOperator{\Ext}{Ext}

\DeclareMathOperator{\gd}{gl.dim}

\DeclareMathOperator{\md}{mod}

\DeclareMathOperator{\Md}{Mod}
\DeclareMathOperator{\proj}{proj}

\DeclareMathOperator{\fl}{fl}
\DeclareMathOperator{\qgr}{qgr}
\DeclareMathOperator{\QGr}{QGr}

\DeclareMathOperator{\thick}{thick}
\DeclareMathOperator{\per}{per}

\DeclareMathOperator{\add}{add}

\DeclareMathOperator{\CM}{CM}
\DeclareMathOperator{\sCM}{\underline{\CM}}

\DeclareMathOperator{\Tot}{Tot}
\DeclareMathOperator{\Totp}{\widehat{\Tot}}
\DeclareMathOperator{\coh}{coh}
\DeclareMathOperator{\silt}{\mathrm{silt}}

\DeclareMathOperator{\ctilt}{\mathrm{ctilt}}

\newcommand\ct[1]{#1\text{-\!}\ctilt}

\DeclareMathOperator{\cone}{cone}

\def\A{\mathcal{A}}
\def\B{\mathcal{B}}
\def\C{\mathcal{C}}
\def\D{\mathcal{D}}

\def\F{\mathcal{F}}

\def\K{\mathcal{K}}
\def\M{\mathcal{M}}

\def\P{\mathcal{P}}
\renewcommand\S{\mathcal{S}}
\def\T{\mathcal{T}}

\def\L{\Lambda}
\def\G{\Gamma}
\def\Del{\Delta}

\def\a{\alpha}

\def\s{\sigma}

\def\Z{\mathbb{Z}}
\def\R{\mathbb{R}}

\def\disoplus{\displaystyle\bigoplus}
\def\fd{\mathrm{fd}}
\def\op{\mathrm{op}}
\def\dg{\mathrm{dg}}

\def\sg{\mathrm{sg}}

\newtheorem{Thm}{Theorem}[section]
\newtheorem{Lem}[Thm]{Lemma}
\newtheorem{Prop}[Thm]{Proposition}
\newtheorem{Cor}[Thm]{Corollary}

\newtheorem{Prop-Def}[Thm]{Proposition-Definition}

\theoremstyle{definition}
\newtheorem{Def}[Thm]{Definition}
\newtheorem{Ex}[Thm]{Example}

\newtheorem{Q}[Thm]{Question}

\newtheorem{Conv}[Thm]{Convention}

\theoremstyle{remark}
\newtheorem{Rem}[Thm]{Remark}

\makeatletter
\renewcommand{\theequation}{\arabic{section}.\arabic{equation}}
\@addtoreset{equation}{section}
\makeatother
\DeclareMathOperator{\Free}{Free}
\DeclareMathOperator{\qper}{qper^\Z\!}
\renewcommand{\P}{\mathbb{P}}

\title[Cluster categories and singularity categories]{Cluster categories of formal DG algebras and \linebreak singularity categories}
\author{Norihiro Hanihara}
\thanks{This work is supported by JSPS KAKENHI Grant Number JP19J21165}
\address{Graduate School of Mathematics, Nagoya University, Furocho, Chikusa-ku, Nagoya, 464-8602, Japan}
\email{m17034e@math.nagoya-u.ac.jp}
\begin{document}
\begin{abstract}
Given a negatively graded Calabi-Yau algebra, we regard it as a DG algebra with vanishing differentials and study its cluster category. We show that this DG algebra is sign-twisted Calabi-Yau, and realize its cluster category as a triangulated hull of an orbit category of a derived category, and as the singularity category of a finite dimensional Iwanaga-Gorenstein algebra. Along the way, we give two results which stand on their own. First, we show that the derived category of coherent sheaves over a Calabi-Yau algebra has a natural cluster tilting subcategory whose dimension is determined by the Calabi-Yau dimension and the $a$-invariant of the algebra. Secondly, we prove that two DG orbit categories obtained from a DG endofunctor and its homotopy inverse are quasi-equivalent. As an application, we show that the higher cluster category of a higher representation infinite algebra is triangle equivalent to the singularity category of an Iwanaga-Gorenstein algebra which is explicitly described. Also, we demonstrate that our results generalize the context of Keller--Murfet--Van den Bergh on the derived orbit category involving a square root of the AR translation.
\end{abstract}
\maketitle
\setcounter{tocdepth}{1}
\tableofcontents
\section{Introduction}
Cluster tilting theory emerged in the beginning of this century in the contexts of categorification of cluster algebras \cite{BMRRT} and higher dimensional Auslander-Reiten theory \cite{Iy07a}, which has lead to fruitful connections between various areas of mathematics. 
A central role is played by the concept of cluster tilting objects in Calabi-Yau (CY) triangulated categories, which gives a categorification of Fomin--Zeleinsky's cluster algebra \cite{CA1}; see for example \cite{Ke10} for an introduction. A general construction of such triangulated categories is given as Amiot's generalized cluster categories \cite{Am09}, which is based on the formalism of differential graded (DG) algebras \cite{Ke06}.
Throughout we fix a field $k$. Recall that a DG $k$-algebra $\L$ is {\it bimodule $(n+1)$-CY} \cite{G} if it is homologically smooth and there is an isomorphism
\[ \RHom_{\L^e}(\L,\L^e)[n+1]\simeq\L \]
in $\D(\L^e)$, where $\L^e$ is the enveloping algebra $\L^\op\otimes_k\L$. Then the {\it cluster category} $\C(\L)$ of $\L$ is the Verdier quotient of the perfect derived category $\per\L$ by the thick subcategory $\D^b(\L)$ consisting of DG modules of finite dimensional total cohomology. The fundamental result due to Amiot and its generalization by Guo \cite{Am09,Guo} states that if $\L$ is a bimodule $(n+1)$-CY DG algebra, then $\C(\L)$ is an $n$-CY triangulated category and $\L\in\C(\L)$ is an $n$-cluster tilting object.


The aim of this paper is to give some descriptions of this cluster category for a certain class of DG algebras, namely {\it formal} DG algebras. Recall that a DG algebra is formal if it is isomorphic to its cohomology in the homotopy category of DG algebras. We therefore start our discussion with a graded (non-DG) algebra $R$ and view it as a DG algebra $R^\dg$ with trivial differentials.

\subsection{Cluster categories and orbit categories}
Our first observation is that we can obtain a class of CY DG algebras from certain graded (non-DG) algebras. Recall the distinct notion of CY algebra for graded non-DG algebras; a graded algebra $R$ over a field $k$ is {\it bimodule $(d+1)$-CY of $a$-invariant $a$} if it is homologically smooth and there is an isomorphism
\[ \RHom_{R^e}(R,R^e)(a)[d+1]\simeq R \]
in the derived category $\D(\Md^\Z\!R^e)$ of graded bimodules (this should not be confused with the derived category $\D((R^\dg)^e)$ of the DG algebra $(R^\dg)^e$). Here $(1)$ is the degree shift functor on the graded modules, while $[1]$ is the suspension in the derived category. Such algebras arise naturally and are studied extensively in representation theory and commutative or non-commutative algebraic geometry \cite{AS,YZ,Boc08,KS08,IR,BS,BSW,MM,AIR,VdB15,RR}.

It is well-known that among CY algebras, those of $a$-invariant $1$ are fundamental in the sense that it is the higher preprojective algebra \cite{IO} of its degree $0$ part \cite{Ke11,MM,HIO,AIR}. Although our results are already non-trivial for $a=1$, we study CY algebras of arbitrary $a$-invariant, which exhibits some additional symmetries.

Let $R$ be a CY algebra. We view it as a DG algebra with vanishing differentials, which we denote by $R^\dg$, and study its properties. Note that the gradings on $R$ and on $R^\dg$ are of different nature; the first one is `algebraic' while the second one is `cohomological' (see \cite[Section 3.1, 15.1]{Ye}). Such homological properties of DG algebras have been investigated for example in \cite{HM,MGYC}.
The following observation shows a relationship between the CY properties of $R$ and $R^\dg$. In particular, we obtain from a graded CY algebra a DG algebra which is always very close to being CY, and often in fact CY. We refer to Theorem \ref{dgcy} for a precise statement. Here we do not need any additional assumptions on $R$ such as (R0) etc below.
\begin{Prop}[Theorem \ref{dgcy}]
Let $R$ be a graded bimodule $(d+1)$-CY algebra of $a$-invariant $a$. Then $R^\dg$ is sign twisted bimodule $(d+a+1)$-CY.
\end{Prop}
For a DG algebra $\L$ satisfying $\per\L\supset\D^b(\L)$, we set
\[ \C(\L):=\per\L/\D^b(\L) \]
and call it, by abuse of language, the {\it cluster category} of $\L$. If $\L$ is a CY DG algebra (e.g. the derived preprojective algebra \cite{Ke11} of a finite dimensional algebra, and the Ginzburg DG algebra \cite{G} associated to a quiver with potential) then $\C(\L)$ is the usual cluster category introduced in \cite{Am09}.
Although our DG algebra $R^\dg$ is not CY in general, it is close enough to CY so that we can define the cluster category $\C(R^\dg)$, which gives rise to a cluster tilting object. To understand this category we first study the categories arising from the graded algebra $R$, and then compare with those arising from $R^\dg$.

We now assume the following on the CY algebra $R$.
\begin{enumerate}
\renewcommand{\labelenumi}{(R\arabic{enumi})}
\setcounter{enumi}{-1}
\item $R$ is negatively graded.
\item Each $R_i$ is finite dimensional.
\end{enumerate}
We note that the condition (R0) can be replaced by positive grading up to Theorem \ref{1dact} below, but {\it negative} grading will be essential in the later discussion.

Let $\per R$ be the perfect derived category of $R$, that is, the thick subcategory of $\D(\Md^\Z\!R)$ generated by the finitely generated graded projective modules. Also let $\D^b(\fl^\Z\!R)$ be the bounded derived category of graded $R$-modules of finite length. We set
\[ \qper R:=\per^\Z\!R/\D^b(\fl^\Z\!R). \]
When $R$ is Noetherian (or more generally graded coherent), we have $\per^\Z\!R=\D^b(\md^\Z\!R)$, the bounded derived category of finitely presented graded $R$-modules. Then the Verdier quotient $\qper R$ is nothing but the derived category of the Serre quotient $\qgr R=\md^\Z\!R/\fl^\Z\!R$, which is regarded as the category of coherent sheaves over the non-commutative projective scheme \cite{AZ} and plays an essential role in non-commutative algebraic geometry. Our category $\qper R$ is thus a generalization of the derived category $\D^b(\qgr R)$.


Our first main result is the existence of a natural cluster tilting subcategory in $\qper R$, which is of independent interest. More importantly, we prove that the construction of $\qper R$ as the Verdier quotient $\per^\Z\!R/\D^b(\fl^\Z\!R)$ lies on the context of Iyama--Yang's formulation \cite{IYa1} of Amiot's cluster category, see Theorem \ref{IYa} and Theorem \ref{dact}, which consequently yields a cluster tilting subcategory.
\begin{Thm}[Theorem \ref{dact}(3)]\label{1dact}
Let $R$ be a graded bimodule $(d+1)$-CY algebra of $a$-invariant $a$ satisfying (R0) and (R1). Then the subcategory
\[ \add\{R(-i)[i]\mid i\in\Z\} \subset \qper R \]
is a $(d+a)$-cluster tilting subcategory.
\end{Thm}
For example, by setting $R$ to be the polynomial ring with standard positive grading, we deduce that the derived category of coherent sheaves over the projective space $\mathbb{P}^d$ has a $(2d+1)$-cluster tilting subcategory $\add\{\mathcal{O}(i)[i] \mid i\in\Z\}$ (cf. \cite{HIO}).

Now we compare the derived categories of the graded algebra $R$ and that of the DG algebra $R^\dg$. An important step is to consider the {\it total module} (see Section \ref{asDG}), which gives a DG functor
\[ \Tot \colon \C(\Md^\Z\!R) \to \C_\dg(R^\dg) \]
on the DG categories of complexes of graded $R$-modules and of DG $R^\dg$-modules, and in turn induces a functor on the derived categories.
We deduce the following result as a consequence of Theorem \ref{1dact} above.
\begin{Cor}[Theorem \ref{reasonable}]\label{1reasonable}
The functor $\Tot$ induces a fully faithful functor
\[ \qper R/(-1)[1] \to \C(R^\dg) \]
whose image generates $\C(R^\dg)$ as a thick subcategory.
\end{Cor}

This is a cluster category analogue of the result in \cite[Theorem 1.3]{KY} for the perfect derived category. Note that this gives a reasonable description of the cluster category since on $\C(R^\dg)$ the degree shift and the suspension is identified, and more accessible in the sense that derived categories are sometimes explicitly described.

Now we apply Minamoto--Mori's equivalence \cite{MM} (see Proposition \ref{mm}); there exists a triangle equivalence $\qper R\simeq\D^b(\md A)$ for the finite dimensional algebra
\begin{equation}\label{eqA}
A=A(R)=\left( 
\begin{array}{cccc}
R_0& 0 & \cdots& 0 \\ 
R_{-1} & R_0&\cdots& 0 \\
\vdots& \vdots&\ddots&\vdots \\
R_{-(a-1)}&R_{-(a-2)}&\cdots&R_0
\end{array}     
\right).
\end{equation}
This algebra is {\it $d$-representation infinite}, which is fundamental in higher dimensional Auslander-Reiten theory \cite{HIO} and non-commutative algebraic geometry \cite{Mi12,MM}.
By the derived equivalence above we deduce that the autoequivalence $\nu_d=-\otimes^L_ADA[-d]$ of $\D^b(\md A)$ has an $a$-th root $\nu_d^{1/a}$ (see (\ref{equp})).  Then we can rewrite Corollary \ref{1reasonable} as a fully faithful functor
\begin{equation}\label{eqroot}
\D^b(\md A)/\nu_d^{-1/a}[1] \hookrightarrow \C(R^\dg)
\end{equation}
whose image generates $\C(R^\dg)$ as a thick subcategory. Note that we can formally write $\nu_d^{-1/a}[1]$ as $\nu_{d+a}^{-1/a}$, thus $\C(R^\dg)$ can be regarded as a `$\Z/a\Z$-quotient' of the $(d+a)$-cluster category of $A$ in the sense that it is obtained from an $a$-th root of the automorphism $\nu_{d+a}$.

The existence of a square root of the AR translation appears in \cite{KMV} and was important in their structure theorem for certain CY categories \cite[Theorem 1.4]{KMV}. Our result (\ref{eqroot}) is an interpretation and a generalization of a situation of their theorem. We discuss in examples (see Example \ref{kmv1} and \ref{kmv3}) how our results specialize to their setting.

\subsection{Cluster categories and singularity categories}
We further describe the cluster category as a singularity category.
Recall that the {\it singularity category} $\D_\sg(\L)$ of a Noetherian ring $\L$ is the Verdier quotient $\D^b(\md\L)/\per\L$, which is widely studied in representation theory and algebraic geometry.
If $\L$ is {\it Iwanaga-Gorenstein} in the sense that the free module $\L$ has finite injective dimension on left and right, then $\D_\sg(\L)$ is canonically equivalent to the stable category $\sCM\L$ of Cohen-Macaulay modules \cite{Bu}.
In the context of cluster tilting theory, Iwanaga-Gorenstein algebras which are stably CY and admit cluster tilting objects, together with the relationship between the cluster categories, have been of particular interest \cite{GLS,Iy07a,KR,IYo,Am09,BIRSc,AIRT,ART,KMV,IO,AO,AIR,TV}.

Let a finite dimensional algebra $A=A(R)$ as in (\ref{eqA}) and an $(A,A)$-bimodule $U$ be
\[ U=U(R)=\left( 
\begin{array}{cccc}
R_{-1} & R_0&\cdots& 0 \\
\vdots& \vdots&\ddots&\vdots \\
R_{-(a-1)}&R_{-(a-2)}&\cdots&R_0 \\
R_{-a}&R_{-(a-1)}&\cdots&R_{-1}
\end{array}     
\right). \]
This is a `relative' $d$-APR tilting of $A$. We have a trivial extension algebra 
\[ B=B(R)=A\oplus U,\]
which turns out to be $d$-Iwanaga-Gorenstein (Proposition \ref{IG}).
Our second main result is a description of the cluster category $\C(R^\dg)$ of a DG algebra in terms of a finite dimensional Iwanaga-Gorenstein algebra $B$.
\begin{Thm}[Theorem \ref{hope}]\label{1hope}
There exists a triangle equivalence
\[ \C(R^\dg) \simeq \D_\sg(B) \]
In particular, $\D_\sg(B)$ is a twisted $(d+a)$-CY category with a $(d+a)$-cluster tilting object.
\end{Thm}
To build the equivalence above, we need a general result on DG orbit categories (Theorem \ref{1orbits} below). Let us explain the connection, for simplicity, in the case $a=1$.

In this case, $R$ is bimodule $(d+1)$-CY of $a$-invariant $1$, thus it is the $(d+1)$-preprojective algebra of its degree $0$ part, which is $A$ in (\ref{eqA}). Then $R^\dg$ is the {\it derived $(d+2)$-preprojective algebra} (or the {\it $(d+2)$-CY completion})
\[ \mathbf{\Pi}_{d+2}(A)=T_A^L\RHom_A(DA,A)[d+1] \]
in the sense of \cite{Ke11}, thus its cluster category $\C(R^\dg)$ is the $(d+1)$-cluster category $\C_{d+1}(A)$ of $A$. 
On the other hand, we have another description of this cluster category $\C_{d+1}(A)$ as a certain singularity category; setting
\[ C=A\oplus DA[-d-2], \]
there exists an equivalence
\[ \C(R^\dg)=\per R^\dg/\D^b(R^\dg) \simeq \thick_{\D(C)}A/\per C \]
by the relative Koszul dual \cite{Am09}. Therefore the equivalence we need is one between the singularity categories 
\[ \thick_{\D(C)}A/\per C\simeq\thick_{\D(B)}A/\per B=\D_\sg(B). \]
Note that they are precisely Keller's description of triangulated hulls \cite{Ke05}, and their equivalence is a consequence of a general equivalence of triangulated hulls, which is our third main result.


Let $\A$ be a pretriangulated DG category, and let $F$, $G$ be DG endofunctors on $\A$ inducing mutually inverse equivalences on $H^0\A$. We then have DG orbit categories $\A/F$ and $\A/G$, whose perfect derived categories give triangulated hulls of $H^0\A/H^0F\simeq H^0\A/H^0G$. 
Our result shows that these triangulated hulls are equivalent.
\begin{Thm}[Theorem \ref{DGorbits}]\label{1orbits}
Suppose there exists a natural transformation $G\circ F \to 1_\A$ inducing a natural isomorphism on $H^0\A$. Then the DG orbit categories $\A/F$ and $\A/G$ are quasi-equivalent. In particular, the triangulated hulls $\per(\A/F)$ and $\per(\A/G)$ are equivalent.
\end{Thm}
We obtain the singular equivalence of $B$ and $C$ by applying this general result to $\A=\C^b(\proj A)$, $F=-\otimes_Ap(\RHom_A(U[1],A))$, and $G=-\otimes_Ap({U}[1])$ (Corollary \ref{typical}), where $p{(-)}$ means a bimodule projective resolution.

As one of applications and examples of our main results, we give a realization of certain higher cluster categories as singularity categories.
\begin{Thm}[Theorem \ref{corpi}]
Any $m$-cluster category of a $d$-representation infinite algebra with $m>d$ is a singularity category of a $d$-Iwanaga-Gorenstein algebra.
\end{Thm}
For example, any (higher) cluster category of a non-Dynkin quiver is the singularity category of a $1$-Iwanaga-Gorenstein algebra. Moreover, we can explicitly describe the Iwanaga-Gorenstein algebra, see Theorem \ref{corpi} and Proposition \ref{Qhat}. This should be compared with the results in \cite{HJ}, where they give a description of higher cluster categories of $1$-representation {finite} algebras (or of Dynkin types) in terms of singularity categories of self-injective algebras, using a combinatorial method.

We also give systematic examples for the case $R$ is a polynomial ring (Section \ref{poly}), and consider examples arising from dimer models (Section \ref{dimer}).
\subsection*{Acknowledgement}
The author is deeply grateful to his supervisor Osamu Iyama for his constant and valuable discussions.

\section{Preliminaries}
We recall some basic concepts on certain structures in triangulated categories. At the end of this section we state Iyama--Yang's result (Theorem \ref{IYa}) which gives a general framework for the construction of `cluster-like' categories. 

Let us start with the following fundamental notion. 
\begin{Def}\label{s}
An object or a subcategory $\M$ in a triangulated category $\T$ is {\it silting} if $\Hom_\T(\M,\M[>\!0])=0$ and $\thick\M=\T$.
\end{Def}
The standard example of a silting object is $\L\in\per\L$ for a negative DG algebra $\L$, that is, a DG algebra with $H^{>0}\L=0$. The same holds for negative DG categories.

Let $\C$ and $\D$ be subcategories of a triangulated category $\T$. We set
\[ \C\ast\D=\{ X\in\T \mid \text{there is a triangle } C \to X \to D \to C[1] \text{ for some } C\in\C,\, D\in\D \}. \]
By the octahedral axiom, the operation $\ast$ is associative. One obtains a co-$t$-structure (or weight structure) \cite{Bo,P} from a silting subcategory, which is given as follows.
\begin{Prop-Def}[{See \cite[Proposition 2.17]{AI}}]
Let $\T$ be an idempotent complete triangulated category with a silting subcategory $\M$. Set
\begin{equation*}
\begin{aligned}
t_{\geq0}&=\bigcup_{l\geq0}\M[-l]\ast\cdots\ast\M[-1]\ast\M, \\
t_{\leq0}&=\bigcup_{l\geq0}\M\ast\M[1]\ast\cdots\ast\M[l]
\end{aligned}
\end{equation*}
Then $(t_{\geq0},t_{\leq0})$ is a co-$t$-structure. We call it the {\rm co-$t$-structure associated to $\M$}.
\end{Prop-Def}
In what follows, we will simply write $t_{\geq0}=\cdots\ast\M$, $t_{\leq0}=\M\ast\cdots$ and so on.
It is important for us that some co-$t$-structures and $t$-structures are related. 
\begin{Def}[\cite{Bo}]
Let $\M\subset\T$ be a silting subcategory and $(t_{\geq0},t_{\leq0})$ the associated co-$t$-structure. Let $t=(t^{\leq0},t^{\geq0})$ be a $t$-structure in $\T$.
\begin{enumerate}
\item We say $t$ is {\it right adjacent to $\M$} if $t_{\leq0}=t^{\leq0}$.
\item We say $t$ is {\it left adjacent to $\M$} if $t_{\geq0}=t^{\geq0}$.
\end{enumerate}
\end{Def}
For example, if $\L$ be a negative DG algebra which is homologically smooth such that each cohomology is finite dimensional, then the standard $t$-structure on $\per\L$ is right adjacent to a silting object $\L\in\per\L$. It follows that its image under the duality $\RHom_\L(-,\L)\colon\per\L\leftrightarrow\per\L^\op$ is left adjacent to a silting object $\L\in\per\L^\op$.

Now let us recall the notion of (relative) Serre functors.
\begin{Def}
Let $\T$ be a $k$-linear $\Hom$-finite triangulated category and $\T^\fd \subset \T$ a thick subcategory.
\begin{enumerate}
	\item An autoequivalence $S\colon\T\to\T$ is a {\it Serre functor} on $\T$ if there is a functorial isomorphism
	\[ \Hom_\T(X,Y) \simeq D\Hom_\T(Y,SX) \]
	for all $X, Y \in \T$.
	\item An autoequivalence $S \colon \T^\fd \to \T^\fd$ is a {\it relative Serre functor} for $\T^\fd\subset\T$ if the above functorial isomorphism holds for all $X\in\T^\fd$ and $Y \in \T$.
	\item We say that $(\T,\T^\fd,S,\M)$ is a {\it relative Serre quadruple} if $S$ is a relative Serre functor for $\T^\fd\subset\T$ and $\M$ is a silting subcategory of $\T$.
\end{enumerate}
\end{Def}
We have one more notion to recall.
\begin{Def}
A $k$-linear category $\C$ is a {\it dualizing variety} if $D=\Hom_k(-,k)$ induces a duality $\md\C\leftrightarrow\md\C^\op$ between the category of finitely presented $\C$-modules.
\end{Def}
For example, the category $\proj\L$ (resp. $\proj^\Z\!\L$) of finitely generated (graded) projective modules over a finite dimensional algebra $\L$ is a dualizing variety.

We are now ready to state the following generalized formulation of Amiot's cluster category.
\begin{Thm}[{\cite{IYa1,IYa2}}]\label{IYa}
Let $(\T,\T^\fd, S, \M)$ be a relative Serre quadruple such that $\M$ is a dualizing variety, and $(t_{\geq0},t_{\leq0})$ be the co-$t$-structure associated to $\M$.
\begin{enumerate}
\item The silting subcategory $\M$ has a right adjacent $t$-structure with $t_{\leq0}^\perp \subset \T^\fd$ if and only if it has a left adjacent $t$-structure with ${}^\perp t_{\geq0} \subset \T^\fd$.
\suspend{enumerate}
Suppose in what follows the equivalent conditions above are satisfied. 
\resume{enumerate}
\item The quotient functor $\pi\colon\T \to \T/\T^\fd$ induces bijections $\Hom_\T(X,Y) \to \Hom_{\T/\T^\fd}(X,Y)$ for all $X \in t_{\leq0}$ and $Y \in St_{\geq2}$. In particular the composition $t_{\leq0}\cap S t_{\geq2} \subset \T \to \T/\T^\fd$ is an additive equivalence.
\item Let $n\geq1$ and assume $\M$ is stable under $S_{n+1}=S\circ[-n-1]$. Then $\pi(\M)\subset\T/\T^\fd$ is an $n$-cluster tilting subcategory.
\end{enumerate}
\end{Thm}
\begin{proof}
	(1) is \cite[Theorem 4.10]{IYa1}. For (2), adapt the proof of \cite[Proposition 5.9]{IYa1}.
	We include a proof of (3) since it requires a modification from \cite[Theorem 5.8]{IYa1}. Since $\M$ is stable under $S_{n+1}$, so is $t_{\geq i}$ for each $i\in\Z$, thus we have $St_{\geq2}=S_{n+1}t_{\geq 1-n}=t_{\geq 1-n}=\cdots\ast\M[n-1]$. Therefore we deduce that the fundamental domain is $\M\ast\cdots\ast\M[n-1]$, hence the result.
\end{proof}
If $\L$ is a bimodule $(n+1)$-CY negative DG algebra with finite dimensional $H^0\L$, then $(\per\L,\D^b(\L),[n+1],\L)$ is a relative Serre quadruple. One can apply the above theorem and recover the original results of Amiot and Guo.

\section{$t$-structure in $\per^\Z\!R$}\label{per}
We will be interested in a negatively graded CY algebra with an $a$-invariant. Before that we will place ourselves in a slightly general setting. Let $R=\bigoplus_{i \leq 0}R_i$ be a {negatively} graded algebra. We assume the following on $R$.
\begin{itemize}
	\item[(R1)] Each $R_i$ is finite dimensional.
	\item[(R2)] Any finite length $R$-module has finite projective dimension.
\end{itemize}
The condition (R2) is clearly satisfied if $R$ is homologically smooth. The aim of this section is to show that there is a $t$-structure in $\per ^\Z\!R$ in this setting.
\begin{Thm}\label{tstr}
Let $R$ be a negatively graded algebra satisfying (R1) and (R2). Set
\begin{equation*}
\begin{aligned}
	t^{\leq0}&=\{ X \in \per^\Z\!R \mid H^i(X) \in \Md^{\leq-i}\!R \text{ for all } i \in \Z \} ,\\
	t^{\geq0}&=\{ X \in \per^\Z\!R \mid H^i(X) \in \Md^{\geq-i}\!R \text{ for all } i \in \Z \}.
\end{aligned}
\end{equation*}
Then $(t^{\leq0},t^{\geq0})$ is a $t$-structure in $\per^\Z\!R$.
\end{Thm}

\begin{Rem}\label{coh}
Under the assumption that $R$ is Noetherian (or more generally graded coherent), there is a version of this for $\D^b(\md^\Z\!R)$ without the `smoothness' condtion (R2), which is in practice far more general than Theorem \ref{tstr} above. We give this general result in Appendix \ref{D}.
\end{Rem}

In the remainder of this section we simply write $\D$ for $\D(\Md^\Z\!R)$. Recall from \cite[Definition 4.1]{AI} that a subcategory $\S$ of a triangulated category $\T$ with arbitrary (set-indexed) coproducts is {\it silting} if it forms a compact set of generators such that $\Hom_\T(A,B[>\!0])=0$ for all $A,B \in \S$. Note that this is a modified version of Definition \ref{s}.

We start our discussion with the following observation.
\begin{Prop}\label{silt}
The subcategory $\M=\add\{R(-i)[i] \mid i \in \Z \} \subset \D$ is silting.
\end{Prop}
\begin{proof}
	Clearly $\M$ is a compact set of generators for $\D$. It remains to show that $\Hom_\D(R,R(-i)[i][j])$ vanishes for each $i \in \Z$ and $j>0$. We only have to consider the case $i=-j<0$, in which case $\Hom_\D(R,R(-i)[i][j])=\Hom^\Z_R(R,R(j))=0$ since $R$ is negatively graded.
\end{proof}

We deduce by \cite[Theorem 4.3]{Ke94} that $\D$ is triangle equivalent to the derived category $\D(\A)$ of a negative DG category $\A$. Then we can consider the standard $t$-structure $(\D_\M^{\leq0},\D_\M^{\geq0})$ associated to $\M$, which is given by
\begin{equation*}
	\begin{aligned}
	\D_\M^{\leq0}&=\{ X \in \D \mid \Hom_{\D}(M,X[i])=0 \text{ for all } M \in \M \text{ and } i>0. \},\\
	\D_\M^{\geq0}&=\{ X \in \D \mid \Hom_{\D}(M,X[i])=0 \text{ for all } M \in \M \text{ and } i<0. \}.
	\end{aligned}
\end{equation*}

Now we use the following computation.
\begin{Lem}\label{henkei}
We have
\begin{equation*}
	\begin{aligned}
	\D_\M^{\leq0}&=\{ X \in \D \mid H^i(X) \in \Md^{\leq-i}\!R \text{ for all } i \in \Z \}, \\
	\D_\M^{\geq0}&=\{ X \in \D \mid H^i(X) \in \Md^{\geq-i}\!R \text{ for all } i \in \Z \}.
	\end{aligned}
\end{equation*}
\end{Lem}
\begin{proof}
	Since $\D_\M^{\leq0}=\M[<\!0]^\perp$, we have
	\[
	\begin{aligned}
	\D_\M^{\leq0}&=\{X \in \D \mid \Hom_\D(R(i)[-i][-j],X)=0 \text{ for all } i \in \Z  \text{ and } j>0 \} \\
	&=\{X \in \D \mid H^{i+j}(X)_{-i}=0 \text{ for all } i \in \Z  \text{ and } j>0 \} \\
	&=\{X \in \D \mid H^i(X)_{-i+j}=0 \text{ for all } i \in \Z  \text{ and } j>0 \},
	\end{aligned}
	\]
	thus the first assertion. By $\D_\M^{\geq0}=\M[>\!0]^\perp$, we similarly have the second equation.
\end{proof}

We need one lemma to ensure that the above $t$-structure in $\D$ restricts to the small derived category.
\begin{Lem}\label{ses}
Consider the truncation triangle $X'\to X\to X'' \to X'[1]$ for $X\in\D$ with $X'\in\D_\M^{\leq0}$, $X''\in\D_\M^{\geq1}$, and let $i\in\Z$.
\begin{enumerate}
\item The triangle induces a short exact sequence
\[ \xymatrix{ 0\ar[r]& H^iX' \ar[r]& H^iX\ar[r]& H^iX''\ar[r]& 0. } \]
\item The above exact sequence is isomorphic to the truncation
\[ \xymatrix{ 0\ar[r]& (H^iX)_{\leq-i} \ar[r]& H^iX\ar[r]& (H^iX)_{>-i}\ar[r]& 0 } \]
of $H^iX\in\Md^\Z\!R$ with respect to the grading.
\end{enumerate}
\end{Lem}
\begin{proof}
	(1)  It is enough to show that the connecting homomorphism $H^iX'' \to H^{i+1}X'$ is $0$ for each $i\in\Z$. Since $X' \in\D_\M^{\leq0}$ and $X''\in\D_\M^{\geq1}$, we have $H^iX''\in\Md^{\geq i+1}\!R$ and $H^{i+1}X\in\Md^{\leq-i-1}\!R$, hence the assertion.\\
	(2)  Similarly, we have $H^iX'\in\Md^{\leq-i}\!R$ and $H^iX''\in\Md^{\geq-i+1}\!R$, thus the exact sequence in (1) has to be the truncation of $H^iX$ with respect to the grading.
\end{proof}

We are now ready to prove the main result. 
\begin{proof}[Proof of Theorem \ref{tstr}]
	We show that the above $t$-structure in $\D$ restricts to $\per^\Z\!R$. Let $X\in\per^\Z\!R$. We have to show that its truncation $X', X''$ in Lemma \ref{ses} are perfect. We may replace $X$ by a bounded complex of finitely generated graded projective $R$-modules. Then we have that $H^iX=0$ for almost all $i$ and that each $H^iX\in\Md^\Z\!R$ is bounded above. Moreover by assumption (R1), each vector space $(H^iX)_j$ is finite dimensional. Then by Lemma \ref{ses}(1) the cohomology $H^iX''$ is $0$ for almost all $i$ and by Lemma \ref{ses}(2) that each $H^iX''\in\Md^\Z\!R$ is bounded below. Therefore $H^iX''$ lies in $\D^b(\fl^\Z\!R)$, hence in $\per^\Z\!R$ by (R2). We conclude that the remaining term $X'$ is also perfect.
\end{proof}

\section{Cluster tilting in $\qper R$ and the $a$-th root of the AR translation}\label{tct}
\subsection{Cluster tilting}
Let us first recall the notion of {\it (twisted) Calabi-Yau algebras}, in the graded case, which is of our central interest. For a graded automorphism $\a$ of a graded ring $\L$, we denote by $(-)_\a$ the twist automorphism on $\Md^\Z\!\L$. 
\begin{Def}\label{cydef}
A graded algebra $R$ is {\it bimodule twisted $n$-Calabi-Yau of $a$-invariant $a$} if it satisfies the following conditions.
\begin{itemize}
	\item $R$ is homologically smooth, that is, $R\in\per^\Z\!R^e$.
	\item There exists a graded automorphism $\a$ of $R$ such that $\RHom_{R^e}(R,R^e)(a)[n]\simeq{}_\a R_1$ in $\D(\Md^\Z\!R^e)$.
\end{itemize}
We refer to $\a$ as the {\it Nakayama automorphism}, which is uniquely determined up to inner automorphism. We say that $R$ is {\it Calabi-Yau} if $\a$ is inner. 
\end{Def}
	
Let $R$ be a negatively graded bimodule twisted $(d+1)$-CY algebra of $a$-invariant $a$ with Nakayama automorphism $\a$. We moreover assume that each $R_i$ is finite dimensional over $k$.

Let us first collect some basic facts on the derived categories $\D^b(\fl^\Z\!R)$, $\per^\Z\!R$, and $\qper R$.
\begin{Prop}
\begin{enumerate}
\item $\D^b(\fl^\Z\!R) \subset \per^\Z\!R$ has a relative Serre functor $(-)_\a(a)[d+1]$.
\item $\qper R$ has a Serre functor $(-)_\a(a)[d]$.
\end{enumerate}
\end{Prop}
\begin{proof}
We include a sketch of the proof for the convenience of the reader. (1) follows from \cite[Lemma 4.1]{Ke08}. To prove (2), we apply \cite[Section 1]{Am09}. For this we construct for each $X, Y \in \per^\Z\!R$ a local $\D^b(\fl^\Z\!R)$-envelope of $X$ relative to $Y$ in the sense of \cite[Definition 1.2]{Am09}. Take a projective resolution $P \to Y$ and pick an integer $n$ such that each term of $P$ is generated in degree $\geq n$. Consider the exact sequence $0 \to X_{<n} \to X \to X_{\geq n} \to 0$ in $\C^b(\Md^\Z\!R)$ obtained by the truncation with respect to the grading on $X$. Then $X_{\geq n} \in \D^b(\fl^\Z\!R)$ and it yields a triangle in $\per^\Z\!R$. Since $\Hom_{\D^b(\Md^\Z\!R)}(Y,X_{<n})=\Hom_{\K^b(\Md^\Z\!R)}(P,X_{<n})=0$, we see that $X \to X_{\geq n}$ gives a local envelope.
\end{proof}

We need the following reformulation of Proposition \ref{silt}.
\begin{Prop}\label{smallsilt}
$\M=\add\{R(-i)[i] \mid i \in \Z \}$ is a silting subcategory of $\per^\Z\!R$.
\end{Prop}
\begin{proof}
	We have seen in Proposition \ref{silt} that $\M$ has no positive self-extensions. Also we clearly have $\thick\M=\per^\Z\!R$.
\end{proof}

Therefore we obtain a relative Serre quadruple $(\per^\Z\!R, \D^b(\fl^\Z\!R),(-)_\a(a)[d+1],\M)$. The first main result of this paper is that this lies on a context of Theorem \ref{IYa}.
\begin{Thm}\label{dact}
Let $R$ be a negatively graded bimodule twisted $(d+1)$-CY algebra of $a$-invariant $a$ such that each $R_i$ is finite dimensional.
\begin{enumerate}
\item $\M$ is a dualizing variety with left and right adjacent $t$-structures.
\item The quotient functor $\pi \colon \per^\Z\!R \to \qper R$ induces bijections $\Hom_{\per^\Z\!R}(X,Y) \to \Hom_{\qper R}(X,Y)$ for each $X\in\M\ast\cdots$ and $Y\in\cdots\ast\M[d+a-1]$. In particular, $\pi$ has a fundamental domain $\M \ast \cdots \ast \M[d+a-1]$.
\item $\pi(\M)=\add\{R(-i)[i] \mid i \in \Z \} \subset \qper R$ is a $(d+a)$-cluster tilting subcategory.
\end{enumerate}
\end{Thm}
In the proof below, we write $\D=\per^\Z\!R$ and $\D^\fd=\D^b(\fl^\Z\!R)$.
\begin{proof} 
(1)  Since $\M=\add\{R(-i)[i]\mid i\in\Z \}$, we have $\M \simeq \proj^\Z\!\L$ with $\L=\bigoplus_{i \in \Z}\Hom_\D(R,R(-i)[i])=R_0$, hence $\M$ is a dualizing variety.

We next show that the silting subcategory $\M\subset\D$ has left and right adjacent $t$-structures. By Theorem \ref{IYa}(1), it suffices to show the existence of the right adjacent $t$-structure with $t_{\leq0}^\perp \subset \D^\fd$. Set $t^{\leq0}=t_{\leq0}$ and $t^{\geq0}=(t^{\leq-1})^\perp$, where $t^{\leq n}=t^{\leq0}[-n]$ and so on.
Note that $t^{\leq0}=\M[<\!0]^\perp$ and $t^{\geq0}=(t_{\leq-1})^\perp=\M[>\!0]^\perp$. Then as in Lemma \ref{henkei} we have
\begin{equation*}
\begin{aligned}
t^{\leq0}&=\{ X \in \D \mid H^i(X) \in \Md^{\leq-i}\!R \text{ for all } i \in \Z \}, \\
t^{\geq0}&=\{ X \in \D \mid H^i(X) \in \Md^{\geq-i}\!R \text{ for all } i \in \Z \}.
\end{aligned}
\end{equation*}
Now the assertion that $(t^{\leq0}, t^{\geq0})$ is a $t$-structure is precisely what we showed in Theorem \ref{tstr}, and clearly $t^{\geq0} \subset \D^\fd$.

(2)(3)  Let $S=(-)_\a(a)[d+1]$ be the relative Serre functor for $\D^\fd\subset\D$. Then $S_{d+a+1}=(-)_\a(a)[-a]$ preserves $\M$. Therefore we have $St_{\geq2}=S_{d+a+1}t_{\geq-d-a+1}=t_{\geq-d-a+1}$, hence (2) by Theorem \ref{IYa}(2), and (3) by Theorem \ref{IYa}(3).
\end{proof}

\subsection{Tilting and the $a$-th root of the AR-translation}
In this subsection we note the result due to Minamoto--Mori \cite{MM}, and give a finite dimensional algebra $A$ which will play a crucial role in the sequel.
Before that let us recall the following notion.
\begin{Def}[\cite{HIO}]
A finite dimensional algebra $\L$ is {\it $d$-representation infinite} if $\gd \L\leq d$ and
\[ \nu_d^{-i}\L\in\md\L \]
holds for all $i\geq0$, where $\nu_d$ is the autoequivalence $-\otimes^L_\L D\L[-d]$ on $\D^b(\md\L)$.
\end{Def}

\begin{Prop}[{\cite[Theorem 4.12]{MM}}]\label{mm}
Let $R$ be a negatively graded bimodule twisted $(d+1)$-CY algebra of $a$-invariant $a$ such that each $R_i$ is finite dimensional.
\begin{enumerate}
\item\label{tilt} $T=\bigoplus_{l=0}^{a-1}R(l)$ is a tilting object in $\qper R$.
\item\label{end} $A=\End_{\qper R}(T)$ is $d$-representation infinite.
\end{enumerate}
Therefore there exists a triangle equivalence $\qper R\simeq\D^b(\md A)$.
\end{Prop}
\begin{proof}
	We will include a proof using Theorem \ref{dact} in Appendix \ref{mmpr}.
	Here we note a complementary discussion to \cite{MM}. Let $\QGr R$ be the Serre quotient of $\Md^\Z\!R$ by the torsion $R$-modules, where a graded $R$-module is torsion if any of its element is annihilated by $R_{\leq n}$ for some $n\leq0$. 
	By \cite{MM} there is a triangle equivalence $\D(\QGr R)\simeq\D(\Md A)$ of big derived categories. Consider its restriction to the thick subcategories of compact objects. By \cite{BV,MM}, $\D(\QGr R)$ is compactly generated by $T$, so the compact objects are $\thick T$ \cite{Ne92}, thus $\qper R$. Similarly the compact objects in $\D(\Md A)$ are $\D^b(\md A)$.
\end{proof}

We are now in the position to state the following important consequence. 
Suppose in what follows that $(-)_\a\simeq1$ on $\qper R$, for example, that $R$ is CY. Let $A$ be the $d$-representation-infinite algebra given in Proposition \ref{mm}, and let $F$ be the autoequivalence on $\D^b(\md A)$ making the diagram below commutative.
\begin{equation}\label{equp}
\xymatrix{ \qper R \ar[r]^-\simeq\ar[d]_{(1)}& \D^b(\md A)\ar[d]^F \\
	\qper R\ar[r]^-\simeq& \D^b(\md A) }
\end{equation}
\begin{Cor}\label{root}
We have $F^a=\nu_d$ as autoequivalences of $\D^b(\md A)$.
\end{Cor}
\begin{proof}
	Comparing the Serre functors on $\qper R\simeq \D^b(\md A)$, the autoequivalences $(a)$ on $\qper R$ and $\nu_d$ on $\D^b(\md A)$ are compatible, hence we obtain the desired result.
\end{proof}
We can therefore regard $F$ as an $a$-th root of the $d$-AR translation $\nu_d$, and denote $F=:\nu_d^{1/a}$, and also $F^{-1}=:\nu_d^{-1/a}$.

Let us give some easy examples of an $a$-th root of the AR translation.
\begin{Ex}\label{easiest}
	Let $R=k[x,y]$ with $\deg x=\deg y=-1$, so $R$ is $2$-CY  of $a$-invariant $2$.	Applying Proposition \ref{mm}, we have a well-known equivalence $\D^b(\qgr R) \simeq \D^b(\md A)$ with $A$ the Kronecker algebra. The AR-quiver of this category looks
	\newcommand{\ara}{\ar@<0.7ex>[ur]\ar@<-0.2ex>[ur]}
	\newcommand{\arb}{\ar@<0.7ex>[dr]\ar@<-0.2ex>[dr]}
	\[ \xymatrix@!R=1mm@!C=1mm{
		&R(1)\arb&&R(-1)\arb&&R(-3)\arb \\
		\cdots\ara&&R\ara&&R(-2)\ara&&\cdots. } \]
	By diagram (\ref{equp}) above, this shows that $\nu_1^{-1/2}$ on $\D^b(\md A)$ acts by `moving one place to the right'.
\end{Ex}
We next look at a higher root of the AR translation.
\begin{Ex}
Let $R=k[x,y]$ with $\deg x=-2$ and $\deg y=-3$, so $R$ is $2$-CY of $a$-invariant $5$. By Proposition \ref{mm}, there is a triangle equivalence $\D^b(\qgr R)\simeq\D^b(\md A)$, where $A$ is the path algebra over $k$ of the following quiver of type $\widetilde{A_4}$:
\[	\begin{tikzpicture}
	\def\radius{1cm} 
	\node (0) at (90:\radius)   {$0$};
	\node (-1) at (18:\radius)    {$1$};
	\node (-2) at (-54:\radius)  {$2$};
	\node (-3) at (-126:\radius) {$3$};
	\node (-4) at (162:\radius)  {$4$};
	
	\begin{scope}[arrows={->[scale=4]}]
	\path[->]
		(0) edge (-2)
		(-2) edge (-4);
		\path[->,font=\small]
		(-1) edge (-3);
		\path[->,font=\small]
		(0) edge (-3);
		\path[->,font=\small]
		(-1) edge (-4);
		\end{scope}
	\end{tikzpicture}
		,\]
with the vertex $i$ corresponding to the summand $R(-i)$. By the triangle equivalence, we see that the AR-quiver of the triangulated category $\D^b(\qgr R)$ has the following connected component.
\[ \xymatrix@!R=1mm@!C=1mm{
	\circ\ar[dr]\ar@{--}[rr]&&\circ\ar[dr]\ar@{--}[rr]&&R(-1)\ar[dr]\ar@{--}[rr]&&\circ\ar[dr]\ar@{--}[rr]&&\circ\ar[dr]\ar@{--}[r]&\\
	&\circ\ar[dr]\ar[ur]&&\circ\ar[dr]\ar[ur]&&R(-4)\ar[dr]\ar[ur]&&\circ\ar[dr]\ar[ur]&&\circ\\
	\circ\ar[dr]\ar[ur]&&\circ\ar[dr]\ar[ur]&&R(-2)\ar[dr]\ar[ur]&&\circ\ar[dr]\ar[ur]&&\circ\ar[dr]\ar[ur]&\\
	&\circ\ar[ur]\ar[dr]&&R\ar[dr]\ar[ur]&&R(-5)\ar[dr]\ar[ur]&&\circ\ar[dr]\ar[ur]&&\circ\\
	\circ\ar[dr]\ar[ur]&&\circ\ar[dr]\ar[ur]&&R(-3)\ar[dr]\ar[ur]&&\circ\ar[dr]\ar[ur]&&\circ\ar[dr]\ar[ur]&\\
	\ar@{--}[r]&\circ\ar[ur]\ar@{--}[rr]&&R(-1)\ar[ur]\ar@{--}[rr]&&\circ\ar[ur]\ar@{--}[rr]&&\circ\ar[ur]\ar@{--}[rr]&&\circ \quad,} \]
where the horizontal ends are identified. We see that $\nu_1^{-1/5}=(-1)$ acts on this component by `moving one place down'.
\end{Ex}
We refer to Section \ref{poly} for more general examples for polynomial rings.

The existence of a square root of the AR translation appears in \cite{KMV} for generalized Kronecker quivers. We show in the following example how to recover their context.
\begin{Ex}\label{kmv1}
	Let $m\geq2$ and set
	\[ R=k\!\left\langle x_1,\ldots,x_m\right\rangle /(x_1^2+\cdots+x_m^2), \quad \deg x_i=-1. \]
	This is a (non-Noetherian) Artin-Schelter regular algebra of dimension $2$ (see \cite{Zh}), thus is twisted CY (\cite[Proposition 4.5]{YZ}, \cite[Theorem 5.15]{RR}).
	Then it is not difficult to deduce that the complex
	\[ \xymatrix{ 0\ar[r]& R\otimes R(2)\ar[r]^-{d_2}& \disoplus_{i=1}^mR\otimes R(1)\ar[r]^-{d_1}& R\otimes R\ar[r]& 0 }\]
	with maps
	\begin{equation*}
	\begin{aligned}
		d_1((1\otimes1)_i)&=x_i\otimes1-1\otimes x_i \\ 
		d_2(1\otimes1)&=\sum_{i=1}^m(x_i\otimes1+1\otimes x_i),
	\end{aligned}
	\end{equation*}
	together with the multiplication map $R\otimes R\to R$ gives a bimodule projective resolution of $R$. Applying $\Hom_{R^e}(-,R^e)$ to this complex shows that $R$ is graded bimodule twisted $2$-CY of $a$-invariant $2$ with Nakayama automorphism $\sigma\colon x_i\mapsto -x_i$.
	
	By Proposition \ref{mm}, we have a derived equivalence $\D^b(\qgr R)\simeq\D^b(\md A)$ with $A=\begin{pmatrix}R_0& 0\\R_{-1}&R_0\end{pmatrix}$, which is the path algebra of the $m$-Kronecker quiver $Q_m=(\bullet\overset{m}{\longrightarrow}\bullet)$. Now the twist automorphism $(-)_\s$ is isomorphic to the identity functor on $\Md^\Z\!R$ so we have an autoequivalence $\nu_1^{-1/2}$ on $\D^b(\md kQ_m)$. The AR quiver of the derived categories has a connected component
	\[ \xymatrix@!R=1mm@!C=1mm{
		&R(1)\ar[dr]|m&&R(-1)\ar[dr]|m&&R(-3)\ar[dr]|m \\
		\cdots\ar[ur]|m&&R\ar[ur]|m&&R(-2)\ar[ur]|m&&\cdots. } \]
	We see that $\nu_1^{-1/2}$ acts by `moving one place right'; compare Example \ref{easiest}.
\end{Ex}

\section{CY algebras as DG algebras}\label{asDG}
We will consider a graded algebra $R$ as a DG algebra with the same underlying graded ring and the vanishing differential. We write $R^\dg$ when considering $R$ as a DG algebra. 

We first collect some sign conventions which is heavily used in this section. 
Throughout this section we denote by $|x|$ the degree of a homogeneous element $x$ in a graded vector space.

\begin{Conv}
Let $\L$ and $\G$ be DG algebras.
\begin{enumerate}
	\item Let $X$ be a DG {\it right} $\L$-module. Then its shift $X[1]$ has the same right $\L$-action as $X$.
	\item Let $X$ be a DG {\it left} $\L$-module. Then its shift $X[1]$ has a left $\L$-action $a\cdot x=(-1)^{|a|}ax$ for $a \in \L$ and $x \in X[1]$.
	\item There is an isomorphism $\cHom_\L(\L[-l],\L) \simeq \L[l]$ of DG left $\L$-modules by $f \mapsto (-1)^{l|f|}f(1)$.
	\item We identify $(\L^e)^\op \simeq \L^e$ via $x \otimes y \leftrightarrow (-1)^{|x||y|}y \otimes x$.
	\item We identify a DG $(\L,\G)$-bimodule $X$ and a DG $\L^\op\otimes_k\G$-module via $\lambda\cdot x\cdot\gamma=(-1)^{|\lambda||x|}x\cdot(\lambda\otimes\gamma)$.
\end{enumerate}
\end{Conv}

We say that a DG algebra $\L$ is {\it twisted bimodule $n$-CY} if it is homologically smooth and there exists a DG automorphism $\s$ of $\L$ such that we have an isomorphism $\RHom_{\L^e}(\L,\L^e)[n] \simeq {}_1\L_\s$ in $\D(\L^e)$. The aim of this section is to note the following observation. Note that the term `CY algebra' means precisely as in Definition \ref{cydef}, and {\it no} additional conditions are imposed.
\begin{Thm}\label{dgcy}
Let $R$ be a graded bimodule $n$-CY algebra of $a$-invariant $a$. Then, $R^\dg$ is twisted bimodule $(n+a)$-CY. Precisely we have the following.
\begin{enumerate}
\item If $a$ is odd, then $R^\dg$ is CY.
\item If $a$ is even, then $R^\dg$ is homologically smooth and $\RHom_{(R^\dg)^e}(R,(R^\dg)^e)[n+a]\simeq{}_1R_\s$ for the automorphism $\sigma \colon x \mapsto (-1)^{|x|}x$ of $R$.
\end{enumerate}
\end{Thm}

Let us introduce some notations. For a DG algebra $\L$, we denote by $\mathrm{C}(\L)=Z^0\C_\dg(\L)$ the (abelian) category of DG $\L$-modules.
Let $X$ be a graded $(R,R)$-bimodule. We can view it as a DG $(R^\dg,R^\dg)$-bimodule with trivial differentials, hence as an $(R^\dg)^e$-module, which gives a functor
\[ (-)^\dg \colon \Md^\Z\!R^e \longrightarrow \mathrm{C}((R^\dg)^e). \] 

We note the following sign-conventional lemmas. The proofs are left to the reader.
\begin{Lem}\label{free}
Let $F=R^e(l)$ be a free $R^e$-module. Then, $F^\dg$ is isomorphic to the free DG $(R^\dg)^e$-module $(R^\dg)^e[l]$. The isomorphism is given by
\[ F^\dg \longrightarrow (R^\dg)^e[l], \quad x\otimes y \mapsto (-1)^{l|x|} x\otimes y. \]
\end{Lem}
\begin{Lem}\label{comm}
Let $X\in\Md^\Z\!R^e$ and $l\in\Z$. Then we have an isomorphism $({}_{\s^l}X(l))^\dg\simeq X^\dg[l]$ in $\mathrm{C}((R^\dg)^e)$.
\end{Lem}

Let us now recall the notion of {total module} of a complex of DG modules. We consider two variations; {\it total sum $\Tot$} and {\it total product $\Totp$}. Let $X=(\cdots\to X^{i-1}\xrightarrow{\delta_X^{i-1}} X^i\xrightarrow{\delta_X^i} X^{i+1}\to\cdots)$ be a complex of DG $\L$-modules, thus each $X^i$ is a DG $\L$-module, each $\delta_X^i$ is a morphism of DG $\L$-modules, and $\delta_X^{i+1}\circ\delta_X^i=0$. Then define
\[ \Tot X=\bigoplus_{i\in\Z}X^i[-i], \qquad \Totp X=\prod_{i\in\Z}X^i[-i] \]
as graded $\L$-modules, and with differentials $\delta_X+\sum_id_{X^i[-i]}$. Here, $\delta_X$ is the differential of the complex $X$ and $d_{X^i[-i]}$ is the differential of the DG module $X^i[-i]$. Then $\Tot X$ and $\Totp X$ are DG $\L$-modules.

For a complex $X=(\cdots\to X^{i-1}\to X^i\to X^{i+1}\to\cdots)$ of DG $\L$-modules and a DG $\L$-module $Y$, we denote by $\cHom_\L(X,Y)$ the complex 
\[ \xymatrix{ \cdots\ar[r]& \cHom_\L(X^{i+1},Y) \ar[r]^-{\cdot \delta_X^i}& \cHom_\L(X^i,Y)\ar[r]^-{\cdot \delta_X^{i-1}}& \cHom_\L(X^{i-1},Y)\ar[r]& \cdots } \]
of DG $k$-modules, with $\cHom_\L(X^i,Y)$ at degree $-i$. 

The following is quite useful for computations.
\begin{Lem}\label{Tot}
Let $\L$ and $\G$ be DG algebras and $X$ a complex of DG $\L$-modules. Then for any DG $(\G,\L)$-bimodule $Y$, we have an isomorphism $\cHom_\L(\Tot X,Y)\simeq \Totp\cHom_\L(X,Y)$ of left DG $\G$-modules.
\end{Lem}
\begin{proof}
	It is easily verified that the degree $n$ part of each side is $\prod_{i \in \Z}\cHom_\L(X^i,Y)^{i+n}$, thus two sides coincide as graded vector spaces. Now we check that this identification is compatible with the differentials and the $\G$-actions. 
	
	On the one hand, the differential on the left-hand-side maps $f \in \cHom_\L(X^i,Y)^{i+n} \subset \cHom_\L(\Tot X,Y)$ to $d_Yf-(-1)^nfd_{\Tot X}=d_Yf-(-1)^nf\delta_X^{i-1}-(-1)^{n+i}fd_{X^i}$.
	
	On the other hand, the differential of the right-hand-side maps $f \in \cHom_\L(X^i,Y)^{i+n} \subset \Totp\cHom_\L(X,Y)$ to $f\delta_X^{i-1}+(-1)^{-i}d_{\cHom_\L(X^i,Y)}f=f\delta_X^{i-1}+(-1)^i(d_Yf-(-1)^{i+n}fd_{X^i})=f\delta_X^{i-1}+(-1)^id_Yf-(-1)^nfd_{X^i}$.
	
	Then one can check that the map $\cHom_\L(\Tot X,Y)\to\Totp\cHom_\L(X,Y)$ given by $f \mapsto (-1)^{i(i+1)/2+in}f$ for $f \in \cHom_\L(X^i,Y)^{i+n}$ is an isomorphism of left DG $\G$-modules.
\end{proof}


An important step for the proof of Theorem \ref{dgcy} is the following observation. We denote by $\Free^\Z\!\L$ the category of (not necessarily finitely generated) graded free modules over a graded ring $\L$.
\begin{Lem}\label{forget}
Let $\s$ be the automorphism $x\mapsto(-1)^{|x|}x$ on $R$. The following two functors are naturally isomorphic.
\begin{enumerate}
	\renewcommand{\labelenumi}{(\alph{enumi})}
	\item $F\colon\Free^\Z\!R^e\xrightarrow{\Hom_{R^e}(-,R^e)}\Md^\Z\!R^e \xrightarrow{{}_\s(-)}\Md^\Z\!R^e$.
	\item $G\colon\Free^\Z\!R^e\xrightarrow{(-)^\dg}{\mathrm{C}}((R^\dg)^e)\xrightarrow{\cHom_{(R^\dg)^e}(-,(R^\dg)^e)}{\mathrm{C}}((R^\dg)^e)\xrightarrow{\rm forget}\Md^\Z\!R^e$.
\end{enumerate}
\end{Lem}
\begin{proof}
	First we define an isomorphism $\varphi_P\colon F(P) \to G(P)$ at the free module $P=R^e(l)$. Clearly we have $F(P)=R\otimes R(-l)$, with the action of $R$ given by $b\cdot(x \otimes y)\cdot a=(-1)^{|b|}xa\otimes by$. On the other hand, using Lemma \ref{free} we see $G(P)=R\otimes R[-l]$, with the $R$-action $b\cdot(x \otimes y)\cdot a=(-1)^{|a|(|b|+|y|)+|b|(l+|x|)}xa\otimes by$ for $x,y \in R$. Now we define an isomorphism $F(P)=R\otimes R(-l) \to R\otimes R(-l)=G(P)$ by the formula
	\[ \varphi_P \colon x \otimes y \mapsto (-1)^{(|x|+l+1)|y|}x \otimes y. \] 
	We can readily check that this is $(R,R)$-bilinear.
	
	Next we show that this isomorphism is natural. Let $P \to Q$ be a morphism in $\Free^\Z\!R^e$. We may assume that this is of the form $R^e(l) \to R^e(m)$ with $1\otimes 1 \mapsto \sum_i u_i\otimes v_i$. Under the functor $(-)^\dg$ and the isomorphism in Lemma \ref{free}, it becomes an $(R^\dg)^e$-linear map $(R^\dg)^e[l]\to (R^\dg)^e[m]$ sending $1\otimes1$ to $\sum_i(-1)^{m|u_i|}u_i\otimes v_i$. Note that we have $|u_i|+|v_i|-m=-l$. Our task is to show that in the diagram below, the middle square is commutative in $\Md^\Z\!R^e$.
	\[ \xymatrix{
		{}_\s\Hom_{R^e}(R^e(m),R^e)\ar[d]\ar@{=}[r]& {}_\s R^e(-m)\ar[d]_f\ar[r]^-{\varphi_Q}& {(R^\dg)^e}[-m]\ar[d]^g& \cHom_{(R^\dg)^e}({(R^\dg)^e}[m],{(R^\dg)^e})\ar@{=}[l]\ar[d] \\
		{}_\s\Hom_{R^e}(R^e(l),R^e)\ar@{=}[r]& {}_\s R^e(-l)\ar[r]^-{\varphi_P}& (R^\dg)^e[-l]& \cHom_{(R^\dg)^e}({(R^\dg)^e}[l],{(R^\dg)^e})\ar@{=}[l] } \]
	By our sign conventions the map $f$ is a left $(R^\dg)^e$-linear map with $1\otimes1\mapsto \sum u_i \otimes v_i$, while $g$ is $1\otimes1 \mapsto (-1)^{m(l+1)}\sum_i(-1)^{m|u_i|}u_i\otimes v_i$. Using these we can verify the desired commutativity.
\end{proof}

We are now ready to prove the main theorem of this section.
\begin{proof}[Proof of Theorem \ref{dgcy}]
	Let $P=(\cdots\to P_1\xrightarrow{} P_0)$ be a free resolution of $R$ in $\C^-(\Md^\Z\!R^e)$. Then the cohomology of $P^\vee=\Hom_{R^e}(P,R^e)$ is concentrated in degree $n$, where it is $R(-a)$. Considering $P$ as a complex $P^\dg$ of DG bimodules as above, the total sum of $P^\dg$ gives an $(R^\dg)^e$-cofibrant resolution of the DG $(R^\dg)^e$-module $R$. Then $\RHom_{(R^\dg)^e}(R,(R^\dg)^e)=\cHom_{(R^\dg)^e}(\Tot P^\dg,(R^\dg)^e)$, which is isomorphic by Lemma \ref{Tot} to the total product of the complex
	\begin{equation*}\label{eqQ}
	Q =( \cHom_{(R^\dg)^e}(P_0^\dg,(R^\dg)^e)\to\cHom_{(R^\dg)^e}(P_1^\dg,(R^\dg)^e)\to\cdots).
	\end{equation*}
	Now by Lemma \ref{forget} this complex is isomorphic to ${}_\s(P^\vee)$ as complexes of graded $(R,R)$-bimodules. Then it has cohomology only at degree $n$, where it is isomorphic to ${}_\s R(-a)$. Therefore $\Totp Q$ is quasi-isomorphic to ${}_\s R(-a)^\dg[-n]$. We conclude by Lemma \ref{comm} that it is $R[-n-a]$ if $a$ is odd, ${}_\s R[-n-a]$ if $a$ is even.	
\end{proof}

\begin{Ex}
Let $R=k[x_1,\cdots,x_n]$ be the polynomial ring.
\begin{enumerate}
\item Set $\deg x_i=-1$ for all $1\leq i\leq n$. Then $R$ is bimodule $n$-CY of $a$-invariant $n$. By Theorem \ref{dgcy}, we see that
\begin{itemize}
	\item $R^\dg$ is $2n$-CY if $n$ is odd.
	\item $R^\dg$ is twisted $2n$-CY if $n$ is even.
\end{itemize}
See Example \ref{fail} below for an illustration in $n=2$ how $R^\dg$ fails to be CY.
\item Set $\deg x_i=1$ for all $1\leq i\leq n$. Then $R$ is bimodule $n$-CY of $a$-invariant $-n$. By Theorem \ref{dgcy}, we see that
\begin{itemize}
	\item $R^\dg$ is $0$-CY if $n$ is odd.
	\item $R^\dg$ is twisted $0$-CY if $n$ is even.
\end{itemize}
This partially recovers \cite[Theorem 6.4]{MGYC}, see also \cite[Example 6.1]{HM}.
\item Set $\deg x_i=2$ for all $1\leq i\leq n$. Then $R$ is bimodule $n$-CY of $a$-invariant $-2n$. By Theorem \ref{dgcy}, we have $\RHom_{(R^\dg)^e}(R,(R^\dg)^e)[-n]\simeq {}_1R_\s$. Note that the automorphism $\s$ is the identity since $R$ is concentrated in even degrees. Therefore we conclude that $R^\dg$ is $(-n)$-CY. This partially recovers \cite[Theorem 6.2]{MGYC}.
\end{enumerate}
\end{Ex}

We explicitly demonstrate how $R^\dg$ fails to be CY.
\begin{Ex}\label{fail}
Let $R=k[x,y]$ with $\deg x=\deg y=-1$. Then the graded ring $R$ is bimodule $2$-CY of $a$-invariant $2$, and has the Koszul resolution, which we depict in the following way.
\[ \xymatrix@R=1mm@C=15mm{& R\otimes R(1)\ar[dr]^{x\otimes 1-1\otimes x}& \\
	R\otimes R(2)\ar[ur]^{-y\otimes1+1\otimes y}\ar[dr]_{x\otimes 1-1\otimes x}&\oplus& R\otimes R \\
	& R \otimes R(1)\ar[ur]_{y\otimes1-1\otimes y} & \qquad, } \]
where the values on the arrows show the image of $1\otimes1$. Now consider this resolution as a complex of DG modules over $S:=(R^\dg)^e$. Under the isomorphism in Lemma \ref{free}, it becomes
\[ \xymatrix@R=1mm@C=15mm{& S[1]\ar[dr]^{x\otimes 1-1\otimes x}& \\
	S[2]\ar[ur]^{y\otimes1+1\otimes y}\ar[dr]_{-x\otimes1-1\otimes x}&\oplus& S \\
	& S[1]\ar[ur]_{y\otimes 1-1\otimes y}& \qquad .} \]
Applying $\cHom_S(-,S)$ we get the complex of left DG $S$-modules
\[ \xymatrix@R=1mm@C=15mm{& S[-1]\ar[dr]^{-y\otimes1-1\otimes y}& \\
	S\ar[ur]^{x\otimes 1-1\otimes x}\ar[dr]_{y\otimes 1-1\otimes y}&\oplus& S[-2] \\
	& S[-1]\ar[ur]_{x\otimes1+1\otimes x}& \qquad ,} \]
whose total module is $\RHom_S(R,S)$ by Lemma \ref{Tot}. We therefore see that $\RHom_S(R,S)[4]\simeq {}_1R_\s$.
\end{Ex}

The appearance of the twist automorphism $\s$ suggests that we should twist the multiplication of the CY algebra $R$ in order that the DG algebra $R^\dg$ to be CY. Let us give an instance where $R$ is twisted CY and $R^\dg$ is CY.
\begin{Ex}\label{kmv2}
Let $m\geq2$ and
\[ R=k\!\left\langle x_1,\ldots,x_m\right\rangle /(x_1^2+\cdots+x_m^2), \quad \deg x_i=-1, \]
If $m=2$, this is a skew polynomial ring with $2$ variables; compare Example \ref{fail} above. For all $m\geq2$, this algebra is Artin-Schelter regular of dimension $2$ (see Example \ref{kmv1}). Recall that the bimodule projective resolution of $R$ is give by the complex
\[ \xymatrix{ 0\ar[r]& R\otimes R(2)\ar[r]^-{d_2}& \disoplus_{i=1}^mR\otimes R(1)\ar[r]^-{d_1}& R\otimes R\ar[r]& 0 }\]
with maps
\begin{equation*}
	\begin{aligned}
		d_1((1\otimes1)_i)&=x_i\otimes1-1\otimes x_i \\ 
		d_2(1\otimes1)&=\sum_{i=1}^m(x_i\otimes1+1\otimes x_i).
	\end{aligned}
\end{equation*}
We claim $R^\dg$ is bimodule $4$-CY.
For this we follow the computation in Example \ref{fail} above. Set $S=(R^\dg)^e$. Applying the functor $(-)^\dg$ and the isomorphism in Lemma \ref{free}, the above complex becomes
\[ \xymatrix{ 0\ar[r]& S[2]\ar[r]^-{d_2}& \disoplus_{i=1}^mS[1]\ar[r]^-{d_1}& S\ar[r]& 0, }\]
where the maps are right $S$-linear morphism such that
\begin{equation*}
	\begin{aligned}
	d_1((1\otimes1)_i)&=x_i\otimes1-1\otimes x_i \\ 
	d_2(1\otimes1)&=\sum_{i=1}^m(-x_i\otimes1+1\otimes x_i).
	\end{aligned}
\end{equation*}
Applying $\cHom_S(-,S)$, we get an isomorphic complex, thus we see that $\RHom_S(R,S)[4]\simeq R$ in $\D(S)$.
\end{Ex}

\section{Cluster categories, derived orbit categories, and singularity categories}\label{thms}
Let $R$ be a CY algebra. We state main results of this paper which describes the cluster category of $R^\dg$ as an orbit category and a singularity category.
\subsection{Cluster categories and orbit categories}
Let $R$ be a negatively graded bimodule $(d+1)$-CY algebra of $a$-invariant $a$ such that each $R_i$ is finite dimensional.
In this subsection, we compare the derived category of $R$ considered as a graded ring, and that of $R$ considered as a DG algebra $R^\dg$ with vanishing differentials. By Theorem \ref{dgcy}, we know that $R^\dg$ is twisted bimodule $(d+a+1)$-CY.

Recall the notion of total module from the previous section. As in the previous section, consider the DG functor
\[ \Tot \colon \C^b_\dg(\Md^\Z\!R) \to \C_\dg(R^\dg). \]
Taking the $0$-th cohomology, it induces a triangle functor $\K^b(\Md^\Z\!R) \to \K(R^\dg)$, which clearly takes acyclic complexes to acyclic DG modules. We therefore obtain a triangle functor
\[ \Tot \colon \per^\Z\!R \to \per R^\dg .\]
Note that this restricts to $\D^b(\fl^\Z\!R) \to \D^b(R^\dg)$, thus it again induces a triangle functor,
\[ \Tot \colon \qper R \to \C(R^\dg). \]

The following result gives a natural and more concrete description of the cluster category of $R^\dg$. Similar type of results for derived or singularity categories are recently obtained in \cite{KY,Bri}.
\begin{Thm}\label{reasonable}
The functor $\Tot \colon \qper R \to \C(R^\dg)$ induces a fully faithful functor
\[ \xymatrix{\qper R/(-1)[1] \ar[r]& \C(R^\dg) }\]
whose image generates $\C(R^\dg)$ as a thick subcategory. Therefore, $\C(R^\dg)$ is a triangulated hull of the orbit category $\qper R/(-1)[1]$.
\end{Thm}
\begin{proof}
Note that the cluster tilting subcategory $\M=\add\{R(-i)[i] \mid i \in \Z \} \subset \qper R$ given in Theorem \ref{dact} is mapped to a cluster tilting object $R \in \C(R^\dg)$, and the functor $\Tot$ induces an equivalence $\M/(-1)[1] \xrightarrow{\simeq}\add R$. Therefore the assertion follows from covering version of the `cluster-Beilinson criterion', see \cite[Lemma 4.5]{KRac}.
\end{proof}

Let $A$ be the $d$-representation infinite algebra given in Proposition \ref{mm} as the endomorphism ring of a tilting object in $\qper R$. Explicitly, we have
\[ A=\left( 
     \begin{array}{cccc}
     R_0& 0 & \cdots& 0 \\ 
     R_{-1} & R_0&\cdots& 0 \\
     \vdots& \vdots&\ddots&\vdots \\
     R_{-(a-1)}&R_{-(a-2)}&\cdots&R_0
     \end{array}     
     \right). \]
Using diagram (\ref{equp}), we deduce the following.
\begin{Cor}\label{maincor}
There exists a fully faithful functor
\[ \xymatrix{\D^b(\md A)/\nu_{d}^{-1/a}[1] \ar[r]& \C(R^\dg) } \]
whose image generates $\C(R^\dg)$ as a thick subcategory. Therefore, $\C(R^\dg)$ is a triangulated hull of the orbit category $\D^b(\md A)/\nu_{d}^{-1/a}[1]$.
\end{Cor}

\subsection{Cluster categories and singularity categories}\label{yokoku}
We present another description of the cluster category $\C(R^\dg)$.
Set
\[ U=\left( 
\begin{array}{cccc}
	R_{-1} & R_0&\cdots& 0 \\
	\vdots& \vdots&\ddots&\vdots \\
	R_{-(a-1)}&R_{-(a-2)}&\cdots&R_0 \\
	R_{-a}&R_{-(a-1)}&\cdots&R_{-1}
\end{array}     
\right), \]
which is an $(A,A)$-bimodule and let $B=A\oplus U$ be the trivial extension algebra. Note that $U=\nu_d^{-1/a}A$ in $\D^b(\md A)$. We have the following basic properties.
\begin{Prop}\label{IG}
	\begin{enumerate}
		\item $U$ is a cotilting bimodule over $A$.
		\item $B$ is a $d$-Iwanaga-Gorenstein algebra.
	\end{enumerate}
\end{Prop}
\begin{proof}
	Since $U=\RHom_{\qper R}(T,T(1))$ for a tilting object $T$ given in Proposition \ref{mm}, it is a cotilting bimodule over $A=\End_{\qper R}(T)$. Then the second assertion follows from \cite[Theorem 4.3]{MY}.
\end{proof}

Now we state the second main result of this paper. 
\begin{Thm}\label{hope}
	There exists a triangle equivalence
	\[ \C(R^\dg) \simeq \D_\sg(B). \]
	In particular, $\D_\sg(B)$ is a twisted $(d+a)$-CY category with a $(d+a)$-cluster tilting object.
\end{Thm}
We postpone the proof of this result to Section \ref{CM} since it requires a general result on DG orbit categories, which we give in the following section.

\subsection{Examples}
Before going on let us give some demonstrations our results. See Section \ref{Cdn}, \ref{poly}, \ref{dimer} for systematic examples.
\begin{Ex}\label{trivial}
Let us start with an almost trivial example. Let
\[ R=k[x], \quad \deg x=-1. \]
This is bimodule $1$-CY of $a$-invariant $1$, thus $R^\dg$ is bimodule $2$-CY by Theorem \ref{dgcy}. We have $A=k$ (which is `$0$-representation infinite') and $U=k$, thus $B=A\oplus U=k[t]/(t^2)$. By Corollary \ref{maincor} and Theorem \ref{hope}, we have equivalences of triangulated categories
\[ \D^b(\md k)/[1]\simeq\C(R^\dg)\simeq \D_\sg(k[t]/(t^2)), \]
which are triangulated categories with $1$ point. See Example \ref{ntrivial} for a generalization, where the case $\deg x=-n$ for arbitrary $n\geq1$ is discussed.
\end{Ex}

\begin{Ex}\label{kmv3}
This is a continuation of Example \ref{kmv1} and \ref{kmv2}.
Let $m\geq2$ and set
\[ R=k\!\left\langle x_1,\ldots,x_m\right\rangle /(x_1^2+\cdots+x_m^2), \quad \deg x_i=-1, \]
which is twisted $2$-CY of $a$-invariant $2$ (Example \ref{kmv1}), and $R^\dg$ is $4$-CY (Example \ref{kmv2}). The $1$-representation infinite algebra $A$ is the path algebra of the $m$-Kronecker quiver $Q_m\colon\bullet\overset{m}{\longrightarrow}\bullet$, and the autoequivalence $\nu_1$ of $\D^b(\md kQ_m)$ has a square root (Example \ref{kmv1}). Also it is not difficult to see that the $1$-Iwanaga-Gorenstein algebra $B$ is presented by the following quiver with relations.
\[ \xymatrix@C=12mm{ \circ\ar@<0.8ex>[r]^{x_1}\ar@{}[r]|{\rotatebox{90}{$\cdot\hspace{-1mm}\cdot\hspace{-1mm}\cdot$}}\ar@<-0.7ex>[r]_{x_m} & \circ\ar@/_17pt/@<-0.1ex>[l]_u}, \quad \textstyle\sum_{i=1}^m(x_iux_i+x_iux_i)=0, \, ux_iu=0 \]
By Corollary \ref{maincor} and Theorem \ref{hope}, there exist triangle equivalences
\[ \D^b(\md kQ_m)/\nu_1^{-1/2}[1] \simeq \D_\sg(B)\simeq \C(R^\dg). \]
\end{Ex}

\begin{Rem}\label{rem1}
	In \cite[Theorem 1.4]{KMV}, Keller--Murfet--Van den Bergh proved that any algebraic $3$-CY triangulated category $\T$ with a $3$-cluster tilting object $T$ such that $\End_\T(T)=k$ and $\Hom_\T(T,T[-1])=k^m$ is triangle equivalent to $\D^b(\md kQ_m)/\tau^{-1/2}[1]$, where $\tau^{-1/2}$ is the square root of the AR translation defined in \cite{KMV} using reflection functors.
	
	Now, since $\C=\C(R^\dg)$ is a $3$-CY category with a cluster tilting object $R$ satisfying $\End_{\C}(R)=k$ and $\Hom_{\C}(R,R[-1])=k^m$, the above equivalent triangulated categories are precisely the $3$-CY category in \cite[Theorem 1.4]{KMV}. Our result shows that the $3$-CY category in \cite{KMV} is also the singularity category of $B$, and can be realized as the cluster category of the DG algebra $R^\dg$.
	We refer to Example \ref{kmv4} for a generalization, which covers the situation in \cite[Remark 3.4.5]{KMV}.
\end{Rem}

\section{Quasi-equivalence of DG orbit categories}\label{DG}
Let $\T$ be a triangulated category and $F \colon \T \to \T$ an autoequivalence. In order to discuss when the orbit category $\T/F$ is triangulated, a triangulated hull of this orbit category was constructed by Keller \cite{Ke05}. The idea was to take an orbit at the level of enhancement of $\T$.

Let $\A$ be a pretriangulated DG category and $F$ a DG endofunctor on $\A$ inducing an equivalence on $H^0\A$. Then the {\it DG orbit category} of $\A$ with respect to $F$, which we denote by $\B=\A/F$, is the DG category with the same objects as $\A$ and with the morphism complex
\[ \B(L,M)=\colim \left( \bigoplus_{n\geq0}\A(F^nL,M) \xrightarrow{F} \bigoplus_{n\geq0}\A(F^nL,FM) \xrightarrow{F} \bigoplus_{n\geq0}\A(F^nL,F^2M) \xrightarrow{F} \cdots \right) \]
for each $L, M \in \B$. It follows that we have $H^0\A/H^0F=H^0\B$, hence $\per\B$ is a triangulated hull of $H^0\A/H^0F$ in the sense that there is a fully faithful functor $H^0\A/H^0F\hookrightarrow\per\B$ whose image generates $\per\B$ as a thick subcategory.

Observe that the above embedding is not dense in general, and thus it is not clear a naive expectation for orbit categories carries over to triangulated hulls. We give an answer to one of such problems.

Let us briefly recall some relevant notions, see \cite[Section 7]{Ke94} for details. Let $\B$ and $\C$ be DG categories. A {\it quasi-functor} $\C \to \B$ is a $(\C,\B)$-bimodule $X$, whose value we denote by $X(B,C)$ for $B \in \B$ and $C \in \C$, such that the DG $\B$-module $X(-,C)$ is isomorphic in $\D(\B)$ to a representable DG $\B$-module for each $C \in \C$. A quasi-functor $X \colon \C \to \B$ is a {\it quasi-equivalence} if $-\otimes^L_\C X \colon\D(\C) \to \D(\B)$ is a triangle equivalence. In this case the equivalence restricts to $\per\C \xrightarrow{\simeq}\per\B$.

\begin{Thm}\label{DGorbits}
Let $\A$ be a pretriangulated DG category, and let $F$, $G$ be DG endofunctors on $\A$ such that $H^0F$ and $H^0G$ are mutually inverse equivalences on $H^0\A$. Suppose that there is a morphism $G\circ F \to 1_\A$ of DG functors inducing a natural isomorphism on $H^0\A$. Then the DG orbit categories $\B=\A/F$ and $\C=\A/G$ are quasi-equivalent. In particular, the triangulated hulls $\per\B$ and $\per\C$ are triangle equivalent.
\end{Thm}

\begin{Rem}\label{typ}
	\begin{enumerate}
		\item We do not need a natural transformation $F \circ G \to 1_\A$.
		\item The assumption on the existence of a natural transformation $G \circ F \to 1_\A$ is satisfied in the following typical case: Let $\L$ be a finite dimensional algebra which is homologically smooth and $\A=\C^b(\proj \L)$. Let $X$ a complex of projective $(\L,\L)$-bimodules such that $F=-\otimes_\L X \colon \A \to \A$ gives an equivalence on $H^0\A=\D^b(\md \L)$. Letting $Y$ be the bimodule projective resolution of $\RHom_\L(X,\L)$, $G=-\otimes_\L Y \colon \A\to\A$ gives an inverse of $F$ on $\D^b(\md\L)$. Moreover, a quasi-isomorphism $Y\otimes_\L X \to \L$ of $(\L,\L)$-bimodule complexes gives a natural transformation $G\circ F \to 1_\A$.
	\end{enumerate}
\end{Rem}

In view of relating the categories $\B$ and $\C$, consider the following direct system of complexes indexed by $\mathbb{N}\times\mathbb{N}$:
\begin{equation}\label{eqU0}
	\xymatrix@C=8mm@R=2.5mm{
		\displaystyle{\bigoplus_{n\geq0}\A(F^nL,M)} \ar[dd]\ar[r]& \displaystyle{\bigoplus_{n\geq0}\A(GF^nL,M)} \ar[dd]\ar[r]& \displaystyle{\bigoplus_{n\geq0}\A(G^2F^nL,M)} \ar[r]\ar[dd]& \cdots \\ \\
		\displaystyle{\bigoplus_{n\geq0}\A(GF^nL,GM)} \ar[dd]\ar[r]& \displaystyle{\bigoplus_{n\geq0}\A(G^2F^nL,GM)} \ar[dd]\ar[r]& \displaystyle{\bigoplus_{n\geq0}\A(G^3F^nL,GM)} \ar[dd]\ar[r]& \cdots \\ \\
		\vdots& \vdots& \vdots& \qquad ,}
\end{equation}
where the vertical transition maps are induced by $G$, and the horizontal ones by $G^{p+1}F^{1+n}L \to G^pF^nL$.

We first fix $q\geq0$ and consider the colimit
\[ U_q(L,M):=\colim_{p\geq0}\bigoplus_{n\geq0}\A(G^{p+q}F^nL,G^qM) \]
of the $(q+1)$-st row.
We can regard $U_q(-,M)$ as a DG $\B$-module for each $M \in \A$ as follows: Let $b \in \B(K,L)$ a morphism represented by $F^mK \to F^rL$ and $u\in U_q(L,M)$ an element represented by $G^{p+q}F^nL \to G^qM$. Enlarging $n$ if necessarily we may assume $r\leq n$. Then define $u\cdot b$ by the composition $G^{p+q}F^{m+n-r}K \xrightarrow{G^{p+q}F^{n-r}b}G^{p+q}F^nL \xrightarrow{u}G^qM$.
Since there is a commutative diagram
\[ \xymatrix@C=10mm{
	G^{p+q}F^{m+n-r}K\ar[rr]^-{G^{p+q}F^{n-r}b}&& G^{p+q}F^nL\ar[r]^-u & G^qM\ar@{=}[d] \\
	G^{p+q+1}F^{1+m+n-r}K\ar[rr]^-{G^{p+q+1}F^{1+n-r}b}\ar[u]&& G^{p+q+1}F^{1+n}L\ar[u]\ar[r] & G^qM } \]
for each $r\leq n$, we see that this action is well-defined.

Let us note the following property of $U_0$.
\begin{Lem}\label{BU}
	The maps $\A(F^n(-),F^pM)\xrightarrow{G^p}\A(G^pF^n(-),G^pF^pM)\to\A(G^pF^n(-),M)$ induce a quasi-isomorphism
	\[ \xymatrix{ \B(-,M)={\colim_{p\geq0}\bigoplus_{n\geq0}\A(F^n(-),F^pM)} \ar[r]& \colim_{p\geq0}\bigoplus_{n\geq0}\A(G^pF^n(-),M)=U_0(-,M) } \]
	of DG $\B$-modules.
\end{Lem}
\begin{proof}
	The commutative diagram
	\[ \xymatrix{
		\A(F^n(-),F^pM)\ar[r]^-{G^p}\ar[d]_F& \A(G^pF^n(-),G^pF^pM)\ar[r]& \A(G^pF^n(-),M)\ar[d] \\
		\A(F^{n+1}(-),F^{p+1}M)\ar[r]^-{G^{p+1}}& \A(G^{p+1}F^{n+1}(-),G^{p+1}F^{p+1}M)\ar[r]& \A(G^{p+1}F^{n+1}(-),M) } \]
	shows the existence of the morphism on the colimts.
	
	It is clear that the induced morphism is a quasi-isomorphism since $F$ and $G$ are mutually inverse equivalences on $H^0\A$. 
\end{proof}

Note that $U_q(-,-)$ constructed above does not have a $\C$-action. The next step toward relating $\B$ and $\C$ is constructing a DG $(\C,\B)$-bimodule which is quasi-isomorphic over $\B$ to $U_0$.

\begin{Lem}\label{Uqis}
	The vertical maps in (\ref{eqU0}) induce a sequence of quasi-isomorphisms
	\[ \xymatrix{ U_0(-,M)\ar[r]& U_1(-,M)\ar[r]& U_2(-,M)\ar[r]& \cdots } \]
	of DG $\B$-modules for each $M \in \A$.
\end{Lem}
\begin{proof}
	Since the vertical maps in (\ref{eqU0}) are quasi-isomorphisms, the induced map $U_q(L,M) \to U_{q+1}(L,M)$ is a quasi-isomorphism for each $q\geq0$, $L \in \B$. It is easily verified that $U_q(-,M) \to U_{q+1}(-,M)$ is $\B$-linear.
\end{proof}

Now define the DG $\B$-module $U(-,M)$ by the colimit
\[ U(-,M):=\colim_{q\geq0}U_q(-,M). \]
For each $L \in \B$ we have $U(L,M)=\colim_{p,q\geq0}\bigoplus_{n\geq0}\A(G^{p+q}F^nL,G^qM)$.

We observe that $U(L,M)$ has a left $\C$-action as follows: Let $c\in\C(M,N)$ be a morphism presented by $G^mM \to G^rN$ and $u\in U(L,M)$ an element presented by $G^{p+q}F^nL\to G^qM$. Enlarging $m$ and $p$ if necessarily we may assume $q\leq m$ and $r\leq p$. Then define $c\cdot u$ by the composition $G^{p+m}F^nL\xrightarrow{G^{m-q}u} G^mM \xrightarrow{c}G^rN$. It is clear that this is well-defined, and compatible with the right $\B$-action. We have therefore obtained a $(\C,\B)$-bimodule $U$.

As a final preparation, we describe an equivalence between the orbit categories $H^0\B$ and $H^0\C$.
\begin{Lem}\label{H0}
	The maps $H^0\A(F^nL,F^pM) \xrightarrow{G^{n+p}} H^0\A(G^{n+p}F^nL,G^{n+p}F^pM) \xrightarrow{\simeq}H^0\A(G^{n+p}F^nL,G^{n}M)\linebreak \xleftarrow{\simeq}H^0\A(G^{p}L,G^{n}M)\to H^0\C(L,M)$ induce an equivalence $H^0\B \simeq H^0\C$.
\end{Lem}

Now we are ready to prove the main theorem of this section.
\begin{proof}[Proof of Theorem \ref{DGorbits}]
	We show that the $(\C,\B)$-bimodule
	\[ U(L,M)=\colim_{p,q\geq0}\bigoplus_{n\geq0}\A(G^{p+q}F^nL,G^qM) \]
	constructed above gives a quasi-equivalence.
	
	For each $M \in \C$ we have quasi-isomorphisms 
	\[ u_M \colon \B(-,M) \to U_0(-,M) \to U(-,M) \]
	of DG $\B$-modules by Lemma \ref{BU} and Lemma \ref{Uqis}, thus $U$ is a quasi-functor.
	
	It remains to show that the induced map
	\[ \Hom_{\D(\C)}(\C(-,L),\C(-,M)) \to \Hom_{\D(\B)}(U(-,L),U(-,M)) \]
	is an isomorphism for each $L, M \in\A$. It suffices to show that the following diagram is commutative:
	\[ \xymatrix{
		\Hom_{\D(\C)}(\C(-,L),\C(-,M)) \ar[r]^{-\otimes^L_\C U}\ar@{=}[d]& \Hom_{\D(\B)}(U(-,L),U(-,M))\ar[r]^\simeq_{u_L}& \Hom_{\D(\B)}(\B(-,L),U(-,M)) \\
		H^0\C(L,M)& \ar[l]_\simeq H^0\B(L,M) \ar@{=}[r]& \Hom_{\D(\B)}(\B(-,L),\B(-,M)) \ar[u]_{\rotatebox{-90}{$\simeq$}}^{u_M}. } \]
	The equivalence $H^0\B \to H^0\C$ given in Lemma \ref{H0} shows that if $f \in H^0\B(L,M)$ is presented by a morphism $F^nL \to F^pM$ in $Z^0\A$, then the corresponding morphism $g \in H^0\C(L,M)$ is presented by $G^pL \to G^nM$ in $Z^0\A$ so that the diagram
	\begin{equation}\label{fg}
		\xymatrix{
			G^{n+p}F^nL \ar[r]^{G^{n+p}f}\ar[d] & G^{n+p}F^pM\ar[d] \\
			G^pL \ar[r]^g& G^nM }
	\end{equation}
	is commutative in $H^0\A$.
	
	Let $f$ and $g$ be of the form above. Then the commutativity we want amounts to saying that $u_M\cdot f=g\cdot u_L$ in $H^0U(L,M)$, where we may regard $\cdot$ as the right action of $\B$ (resp. left action of $\C$) on $U$. Since $u_L \in U(L,L)$ is presented by the identity morphism in $\A$ the element $g \cdot u_L \in U(L,M)$ is presented by $G^{p+n}F^nL \to G^pL \xrightarrow{1} G^pL \xrightarrow{g} G^nM$.
	Similarly, since $u_M \in U(M,M)$ is presented by the identity morphism in $\A$, the element $u_M\cdot f\in U(L,M)$ is presented by $G^pF^nL \xrightarrow{G^pf} G^pF^pM \to M$, hence by one obtained by applying $G^n$.
	Then the commutativity of the diagram (\ref{fg}) in $H^0\A$ implies that we have $u_M\cdot f=g\cdot u_L$ in $H^0U(L,M)$.
\end{proof}

Let us apply our general result to the setting of finite dimensional algebras. Let $\L$ be a finite dimensional algebra which is homologically smooth, and $X$ a complex of $(\L,\L)$-bimodules such that $F=-\otimes^L_\L X$ gives an autoequivalence on $\D^b(\md\L)$. We assume that for each $L, M\in\D^b(\md\L)$, we have $\Hom_{\D(\L)}(L,F^iM)=0$ for almost all $i\in\Z$.

Let $\G$ be the trivial extension DG algebra of $\L$ by $X[-1]$, that is, $\G=\L\oplus X[-1]$ as a complex, and the multiplication is given by the bimodule structure on $X[-1]$. Then Keller's theorem \cite[Theorem 2]{Ke05} gives the equivalence
\[ \xymatrix{ \per\B \ar[r]^-\simeq & \thick_{\D(\G)}\L/\per \G } \]
which is compatible with the natural functors from $\D^b(\md\L)$.

We have the same equivalence arising from the inverse of $F$; Let $Y$ be a bimodule projective resolution of $\RHom_\L(X,\L)$ and $G=-\otimes_\L Y \colon \A \to \A$, which induces a quasi-inverse to $F=-\otimes_\L X$ on $\D^b(\md\L)$. Let $\C=\A/G$ be the DG orbit category and $\Del=\L\oplus Y[-1]$ be the trivial extension DG algebra so that we have an equivalence
\[ \xymatrix{ \per\C \ar[r]^-\simeq & \thick_{\D(\Del)}\L/\per \Del } \]
By the above equivalences and Theorem \ref{DGorbits} we deduce the following consequence.
\begin{Cor}\label{typical}
	There exists a triangle equivalence
	\[ \xymatrix{ \thick_{\D(\G)}\L/\per \G \ar[r]^\simeq & \thick_{\D(\Del)}\L/\per \Del } \]
	which is compatible with the projection functors from $\D^b(\md \L)$.
\end{Cor}

Observe that the above result gives a certain singular equivalence of DG algebras $\G$ and $\Del$. We end this section by posing the following question.
\begin{Q}
	Is it possible to describe the equivalence in Corollary \ref{typical} directly?
\end{Q}

\section{Proof of Theorem \ref{hope}}\label{CM}
We now give a proof of a main result Theorem \ref{hope} of this paper using the result from the previous section. In fact, the essential part of the proof does not depend on our specific setup, so we first state the intermediate result in Proposition \ref{BS} below, which can also be viewed as a DG version of Theorem \ref{hope}.

Let $\L$ be a finite dimensional algebra which is homologically smooth, $X$ a complex of $(\L,\L)$-bimodules such that $F=-\otimes^L_\L X$ gives an autoequivalence on $\D^b(\md\L)$.
We assume the following on the tilting complex $X$:
\begin{enumerate}
	\renewcommand{\labelenumi}{(X\arabic{enumi})}
	\item For each $L, M\in\D^b(\md\L)$, we have $\Hom_{\D(\L)}(L,F^iM)=0$ for almost all $i\in\Z$. 
	\item $X$ is concentrated in degree $\leq0$.
\end{enumerate}
Let
\[ \G=\L\oplus X[-1], \quad \Sigma=T^L_\L X \]
be the trivial extension DG algebra, and respectively the derived tensor algebra, that is, the tensor algebra of a bimodule projective resolution of $X$. Then the conditions imply that $\Sigma$ is a negative DG algebra and its cohomology $H^0\Sigma=T_\L (H^0X)$ is finite dimensional.

Recall from the introduction that we have denoted by
\[ \C(\Pi)=\per\Pi/\D^b(\Pi) \]
and called it the cluster category for any DG algebra $\Pi$ satisfying $\per\Pi\supset\D^b(\Pi)$.  
Our general intermediate result is an equivalence between the cluster category of the tensor algebra and the singularity category of the trivial extension algebra.
\begin{Prop}\label{BS}
The DG algebra $\Sigma$ satisfies $\per\Sigma\supset\D^b(\Sigma)$, and there exists a triangle equivalence
\[ \C(\Sigma)\simeq\thick_{\D(\G)}\L/\per\G. \]
\end{Prop}

The first step is to apply Corollary \ref{typical}. Let $Y$ be a bimodule projective resolution of $\RHom_\L(X,\L)$ and set
\[ \Del=\L\oplus Y[-1]. \]
Then we have a triangle equivalence
\begin{equation}\label{DGinv}
	\thick_{\D(\G)}\L/\per\G \simeq \thick_{\D(\Del)}\L/\per\Del.
\end{equation}

We next use the following computation of a DG endomorphism algebra.
\begin{Lem}[{See \cite[Lemma 4.13]{Am09}}]
There exists an isomorphism $\RHom_\Del(\L,\L) \simeq T^L_\L X=\Sigma$ in the homotopy category of DG algebras, that is, these two DG algebras are related by a zig-zag of quasi-isomorphisms of DG algebras.
\end{Lem}

We also need the equivalence by relative Koszul dual.
\begin{Lem}\label{Ami}
The functor $\RHom_\Del(\L,-)\colon\D(\Del) \to \D(\Sigma)$ restricts to equivalences $\thick_{\D(\Del)}\L\simeq\per \Sigma$ and $\per \Del \simeq \D^b(\Sigma)$.
Therefore we have $\per \Sigma\supset\D^b(\Sigma)$ and an equivalence $\thick_{\D(\Del)}\L/\per \Del \xrightarrow{\simeq} \C(\Sigma)$.
\end{Lem}
\begin{proof}
	The first assertion is clear. We prove second one. Since $\Del=\RHom_\L(\Del,Y[-1])$ as $(\L,\Del)$-bimodules, we have $\RHom_\Del(\L,\Del)=\RHom_\Del(\L,\RHom_\L(\Del,Y[-1]))=\RHom_\L(\L,Y[-1])=Y[-1]$, which has finite dimensional total cohomology. Therefore the functor maps $\per \Del$ into $\D^b(\Sigma)$. It remains to show the essential surjectivity. Since $\Sigma$ is a negative DG algebra, the finite dimensional derived category $\D^b(\Sigma)$ has a bounded $t$-structure whose heart is equivalent to the category of finite dimensional modules over $H^0\Sigma$, hence it suffices to show that the heart is contained in the image of the functor. Note that $H^0\Sigma=T_\L(H^0X)$ is a finite dimensional graded algebra whose degree $0$ part $\L$ has finite global dimension. Therefore it is sufficient to prove that $D\L$ is in the image. Clearly $D\L=\RHom_\Del(\L,D\Del)$, so it remains to show $D\Del\in\per \Del$. But we have $D\Del=\RHom_\L(\Del,D\L)$ and $D\L\in\thick_{\D(\L)}Y[-1]$, hence the assertion.
\end{proof}

\begin{proof}[Proof of Proposition \ref{BS}]
	It is a consequence of (\ref{DGinv}) and Lemma \ref{Ami}.
\end{proof}

We now return to our setup from Section \ref{thms}. Recall that $R$ is a negatively graded bimodule $(d+1)$-CY algebra of $a$-invariant $a$ such that each $R_i$ is finite dimensional, and that $A$ is a $d$-representation infinite algebra in Proposition \ref{mm} given by
\[ A=\left( 
\begin{array}{cccc}
R_0& 0 & \cdots& 0 \\ 
R_{-1} & R_0&\cdots& 0 \\
\vdots& \vdots&\ddots&\vdots \\
R_{-(a-1)}&R_{-(a-2)}&\cdots&R_0
\end{array}     
\right), \]
whose $d$-AR-translation $\nu_d$ has an $a$-th root defined by diagram (\ref{equp}). We have a cotilting bimodule $U=\nu_d^{-1/a}A$ and a $d$-Iwanaga-Gorenstein algebra $B=A\oplus U$ (see Proposition \ref{IG}).

Let us first apply Proposition \ref{BS}. Set $\L=A$ and $X=U[1]$. Since $U$ is a preprojective module over a $d$-representation infinite algebra, it clearly satisfies (X1) and (X2). Then the DG algebra $\G=\L\oplus X[-1]$ concentrates in degree $0$, and is nothing but our Iwanaga-Gorenstein algebra $B=A\oplus U$. Moreover, we see that the DG algebra $\Sigma$ is the tensor algebra 
\[ S:=T_A(U[1]) \]
with trivial differentials. Proposition \ref{BS} for this case gives the following equivalence.
\begin{Cor}\label{eqBS}
There exists a triangle equivalence
\[ \C(S)\simeq\D_\sg(B). \]
\end{Cor}


To compare the cluster categories of $S$ and $R^\dg$, we need another intermediate DG algebra, which is
\[ \widetilde{S}:=\cEnd_R(T), \text{ with } T=R \oplus R[-1] \oplus \cdots\oplus R[-(a-1)]. \]

Let us first state an easy relationship between $R^\dg$ and $\widetilde{S}$.
We say that DG algebras $\Pi_1$ and $\Pi_2$ are {\it DG Morita equivalent} if there is a $(\Pi_1,\Pi_2)$-bimodule $X$ such that $-\otimes^L_{\Pi_1} X$ induces an equivalence $\D(\Pi_1) \xrightarrow{\simeq}\D(\Pi_2)$, or equivalently there exists a compact generator $M \in \D(\Pi_2)$ whose DG endomorphism algebra is quasi-equivalent to $\Pi_1$ \cite[Theorem 8.2]{Ke94}\cite[Theorem 3.11]{Ke06}.

We immediately have the following lemma.
\begin{Lem}\label{RStilde}
\begin{enumerate}
\item The DG algebras $R^\dg$ and $\widetilde{S}$ are DG Morita equivalent.
\item We have $\per \widetilde{S}\supset \D^b(\widetilde{S})$.
\item The cluster categories of $R^\dg$ and $\widetilde{S}$ are equivalent.
\end{enumerate}
\end{Lem}
\begin{proof}
	Since $\widetilde{S}$ is the DG endomorphism ring of a compact generator $T \in \D(R^\dg)$, we have (1). Then the assertions (2) and (3) follow.
\end{proof}

We next discuss the relationship between $S$ and $\widetilde{S}$.
\begin{Lem}\label{SStilde}
\begin{enumerate}
\item $S$ is a finite codimensional subalgebra of $\widetilde{S}$.
\item The cluster categories of $S$ and $\widetilde{S}$ are equivalent.
\end{enumerate}
\end{Lem}
\begin{proof}
	(1)  This is a consequence of $S=\bigoplus_{i\geq0}\Hom_{\D(A)}(A,U^i)$ and $\widetilde{S}=\bigoplus_{i\in\Z}\Hom_{\D(A)}(A,U^i)$.\\
	(2)  Consider the pair of adjoint functors $F=-\otimes^L_S \widetilde{S} \colon \D(S) \to \D(\widetilde{S})$ and $G=\mathrm{res} \colon \D(\widetilde{S}) \to \D(S)$. 
	
	{\it Step 1: Restrictions of the adjoint pair.}\quad
	We observe that these functors restrict to the perfect and finite dimensional derived categories. Clearly $-\otimes^L_S\widetilde{S}$ restricts to $\per S \to \per\widetilde{S}$, and $\mathrm{res}$ to $\D^b(\widetilde{S})\to\D^b(S)$. Also since $\widetilde{S}$ is perfect over $S$ by (1) and Lemma \ref{Ami}, the remaining assertions follow.
	Therefore the above functors induce an adjoint pair between the Verdier quotients. 
	
	{\it Step 2: The unit and counit maps.}\quad
	We show that the unit and counit maps are isomorphisms. 
	Let $X \in \C(S)$ and consider the unit map $u_X\colon X \to X \otimes^L_S\widetilde{S}$. Since this is obtained by applying $X \otimes^L_S-$ to an isomorphism $S \to \widetilde{S}$ in $\C(S)$, it is an isomorphism. Next let $Y \in \C(\widetilde{S})$ and $v_Y \colon FGY \to Y$ the counit. Note that $G$ detects isomorphisms. Indeed $v_Y$ is an isomorphism in $\C(\widetilde{S})$ if and only if, as a map in $\per\widetilde{S}$, the cone of $v_Y$ is in $\D^b(\widetilde{S})$. But this property is detected by the restriction functor $G$.
	Now the claim $G(v_Y)$ is an isomorphism follows from the fact that the composition $GY\xrightarrow{u_{GY}}GFGY\xrightarrow{G(v_Y)}GY$ equals the identity,	which is a general property of an adjoint pair, and the isomorphism of the unit.
\end{proof}
	
Now, Theorem \ref{hope} is a consequence of the following sequence of equivalences in the upper row.
\begin{equation*}
	\xymatrix@R=2mm@C=9mm{
	\D_\sg(B)\ar@{=}[rr]^{\text{Cor \ref{eqBS}}}\ar@{=}[dr]_{\text{Cor \ref{typical}\qquad}}&& \C(S)\ar@{=}[rr]^{\text{Lem \ref{SStilde}}}&& \C(\widetilde{S})\ar@{=}[rr]^{\text{Lem \ref{RStilde}}}&& \C(R^\dg) \\
	& \thick_{\D(C)}A/\per C\ar@{=}[ur]_{\qquad \text{Lem \ref{Ami}}} }
\end{equation*}
Here we have included $C=A\oplus\RHom_A(U[1],A)[-1]$, which is the DG algebra $\Del$ in Proposition \ref{BS} for our specific setup.

\bigskip
We record the following formula for the $d$-representation infinite algebra $A$, the cotilting bimodule $U$, and the $d$-Iwanaga-Gorenstein algebra $B$ which are determined by $R$.
\begin{equation}\label{eqAUB}
A=\left( 
\begin{array}{cccc}
R_0& 0 & \cdots& 0 \\ 
R_{-1} & R_0&\cdots& 0 \\
\vdots& \vdots&\ddots&\vdots \\
R_{-(a-1)}&R_{-(a-2)}&\cdots&R_0
\end{array}     
\right), \quad
U=\left( 
\begin{array}{cccc}
R_{-1} & R_0&\cdots& 0 \\
\vdots& \vdots&\ddots&\vdots \\
R_{-(a-1)}&R_{-(a-2)}&\cdots&R_0 \\
R_{-a}&R_{-(a-1)}&\cdots&R_{-1}
\end{array}     
\right), \quad
B=A\oplus U.
\end{equation}
Let us also record the equivalences we have shown.
\begin{equation}\label{eqsumm}
\xymatrix{ \D^b(\md A)/\nu_d^{-1/a}[1]\ar@{^(->}[r]& \D_\sg(B)\ar[r]^-\simeq& \C(R^\dg) }
\end{equation}

We end this section with the obvious lemma, which is useful for later computation.
\begin{Lem}\label{comp}
	Let $J_0$ be the Jacobson radical of $R_0$.
	\begin{enumerate}
		\item The Jacobson radical $J_A$ of $A$ is $\begin{pmatrix}J_0&0&\cdots&0\\R_{-1}&J_0&\cdots&0\\\vdots&\vdots&\ddots&\vdots\\R_{-(a-1)}&R_{-(a-2)}&\cdots&J_0\end{pmatrix}$.
		\item The Jacobson radical $J_B$ of $B$ is $J_A\oplus U$.
		\item We have $J_B/J_B^2=J_A/J_A^2\oplus U/(J_AU+UJ_A)$, with
		\[ U/(J_AU+UJ_A)=\left(
		\begin{array}{cccc}
		0& R_0/J_0&\cdots& 0 \\
		\vdots& \vdots&\ddots&\vdots \\
		0&0&\cdots&R_0/J_0 \\
		\frac{R_{-a}}{J_0R_{-a}+R_{-a}J_0+\sum_{i+j=a, i,j>0}R_{-i}R_{-j}}&0&\cdots&0
		\end{array}     
		\right). \]
	\end{enumerate}
\end{Lem}

\section{Higher cluster categories of higher representation infinite algebras}\label{Cdn}
We give an application of Theorem \ref{hope}, which is given by taking the CY algebra $R$ as a (higher) preprojective algebra. We prove that any $m$-cluster category of a $d$-representation infinite algebra with $m>d$ is the singularity category of a $d$-Iwanaga-Gorenstein algebra. Moreover we explicitly describe the quiver and relations of the Iwanaga-Gorenstein algebra for the case $d=1$, that is, when $A$ is hereditary. 
\begin{Thm}\label{corpi}
Let $A$ be a $d$-representation infinite algebra and fix $n\geq1$. Let $U=\Ext_A^d(DA,A)$ and
\[ B=\left( 
\begin{array}{cccc}
	A & 0&\cdots& 0 \\
	0&A&\cdots&0\\
	\vdots& \vdots&\ddots&\vdots \\
	0&0&\cdots&A
\end{array}     
\right)\oplus
\left( 
\begin{array}{cccc}
	0 & A&\cdots& 0 \\
	\vdots& \vdots&\ddots&\vdots \\
	0&0&\cdots&A\\
	U&0&\cdots&0
\end{array}     
\right) \]
the trivial extension algebra, where the matrix is $n\times n$.
Then $B$ is $d$-Iwanaga-Gorenstein and there exists a triangle equivalence
\[ \C_{d+n}(A)\simeq\D_\sg(B). \]
\end{Thm}
\begin{proof}
	Let $\Pi=T_AU$ be the $(d+1)$-preprojective algebra of $A$. Give a grading on $\Pi$ by setting $\deg U=-n$ so that $\Pi$ is a bimodule $(d+1)$-CY algebra of $a$-invariant $n$. Note that $\Pi^\dg$ is quasi-isomorphic to the derived $(d+n+1)$-preprojective algebra (or the $(d+n+1)$-Calabi-Yau completion) $T_A^L\RHom_A(DA,A)[d+n]$, in the sense of \cite{Ke11}. Therefore the cluster category $\C(\Pi^\dg)$ is nothing but the $(d+n)$-cluster category $\C_{d+n}(A)$ of $A$. By Theorem \ref{hope}, we have a triangle equivalence $\C(\Pi^\dg)\simeq\D_\sg(B(\Pi))$ for the $d$-Iwanaga-Gorenstein algebra $B(\Pi)$ in (\ref{eqAUB}) for the Calabi-Yau algebra $\Pi$, which is precisely the $B$ above in the statement.
\end{proof}

\begin{Rem}
We give a general discussion in Appendix \ref{nbai} the effect of `multiplying gradings' as we did in the above proof. See Corollary \ref{Bhat} for a description as a derived orbit category, which predicts the above equivalence.
\end{Rem}

\subsection{The case $d=1$.}
Let us record the special case $d=1$, that is, when $A$ is hereditary.
\begin{Cor}\label{d=1}
Let $Q$ be a finite connected acyclic non-Dynkin quiver, $A=kQ$ and fix $n\geq1$. Then we have a triangle equivalence
\[ \C_{n+1}(kQ)\simeq\D_\sg(B) \]
for the $1$-Iwanaga-Gorenstein algebra $B$ in Theorem \ref{corpi}.
\end{Cor}
In this case we can explicitly describe the quiver and relations for $B$.
\begin{Prop}\label{Qhat}
The $1$-Iwanaga-Gorenstein algebra ${B}$ in Theorem \ref{corpi} is presented by the quiver $\widehat{Q}$ with 
\begin{enumerate}
\renewcommand{\labelenumi}{(\alph{enumi})}
\renewcommand{\labelenumii}{(\roman{enumii})}
	\item vertices $Q_0\times\{1,\ldots,n\}$,
	\item three kinds of arrows
	\begin{enumerate}
		\item $a=a^l\colon(i,l)\to(j,l)$ for each $a\colon i\to j$ in $Q_1$ and $1\leq l\leq n$.
		\item $v=v_i^l\colon(i,l+1)\to(i,l)$ for each $i\in Q_0$ and $1\leq l< n$.
		\item $a^\ast\colon(j,1)\to(i,n)$ for each $a\colon i\to j$ in $Q_1$.
	\end{enumerate}
	\item three kinds of relations
	\begin{enumerate}
		\item $a^{l}v_i^{l}=v_j^{l}a^{l+1}$ for each $a\colon i\to j$ in $Q_1$ and $1\leq l<n$.
		\item $\sum_{s(a)=i}a^\ast a^1=\sum_{t(a)=i}a^la^\ast$ for all $i\in Q_0$.
		\item $v_i^{l-1}v_i^l=0$, $a^\ast v=0$, $va^\ast=0$ if $n\geq2$, and $a^\ast bc^\ast=0$ for any composable $a, b, c \in Q_1$ if $n=1$.
	\end{enumerate}
\end{enumerate}
\end{Prop}
We use a reformulation of a well-known fact on preprojective algebras. We denote by $J_\L$ the Jacobson radical of a ring $\L$.
\begin{Lem}\label{ast}
Let $Q$ be a finite acyclic quiver, $A=kQ$, and $U=\Ext_{A}^1(DA,A)$. Then there exists a subset $\{u(a^\ast)\mid a\in Q_1\}$ of $U$ which gives a basis of $U/(J_AU+UJ_A)$ as a $(kQ_0,kQ_0)$-bimodule, and such that $\sum_{a\in Q_1}(au(a^\ast)-u(a^\ast) a)=0$ in $U$.
\end{Lem}
\begin{proof}
	Consider the two presentations $T_AU=k\overline{Q}/(\sum_{a\in Q_1}(aa^\ast-a^\ast a))$ of the preprojective algebra $\Pi$ of $Q$, where $\overline{Q}$ is the double quiver of $Q$ obtained by adding the opposite arrows $\{a^\ast\colon j\to i\mid a\colon i\to j \text{ in } Q \}$. Take the elements of $U$ corresponding to $\{a^\ast\mid a\in Q_1\}\subset\overline{Q}_1$. This is a desired set since we have $U/(J_AU+UJ_A)=kQ_1^\ast$ as $(kQ_0,kQ_0)$-bimodules by looking at the degree $1$ part of $J_\Pi/J_\Pi^2$.
\end{proof}

\begin{proof}[Proof of Proposition \ref{Qhat}]
	By Lemma \ref{comp}, we have
	\[ J_B/J_B^2=
	\left( 
	\begin{array}{cccc}
	J_A/J_A^2& 0&\cdots& 0 \\
	0&J_A/J_A^2&\cdots&0\\
	\vdots& \vdots&\ddots&\vdots \\
	0&0&\cdots&J_A/J_A^2
	\end{array}     
	\right)\oplus
	\left(
	\begin{array}{cccc}
	0& A/J_A&\cdots& 0 \\
	\vdots& \vdots&\ddots&\vdots \\
	0&0&\cdots&A/J_A \\
	{U}/{(J_AU+UJ_A)}&0&\cdots&0
	\end{array}     
	\right). \]
	Therefore we see that the quiver of $B$ consists of the following.
	\begin{itemize}
		\item $n$ copies $Q^1, \ldots, Q^n$ of $Q$.
		\item The arrows from $Q^{l+1}$ to $Q^{l}$ corresponding to the idempotents in $A$.
		\item The arrows from $Q^1$ to $Q^n$ corresponding to the basis of $U/(J_AU+UJ_A)$ as $(kQ_0,kQ_0)$-bimodules.
	\end{itemize}
	In view of Lemma \ref{ast}, these vertices and arrows are precisely the ones described in (a) and (b), hence the quiver of $B$ is $\widehat{Q}$.
	
	Now we determine the relations. To simplify the discussion below, we give a grading on $B$ in Theorem \ref{corpi} by setting the first factor to have degree $0$, and the second one to have degree $1$. 
	Similarly, give a grading on $\widehat{Q}$ by setting the arrows in (i) to have degree $0$, and in (ii), (iii) to have degree $1$.
	Take a subset $\{u(a^\ast)\mid a\in Q_1\}$ of $U$ given in Lemma \ref{ast} and consider the map $\widehat{Q}\to B$ defined by
	\begin{itemize}
		\item the natural embedding $kQ^l\to A\times\cdots\times A=B_0$ into the $l$-th factor,
		\item $v_i^l\mapsto e_i$, where $e_i$ is the corresponding idempotent of $A$ in the $(l+1,l)$-component of $B_1$,
		\item $a^\ast\mapsto u(a^\ast)$, where $u(a^\ast)\in U$ is in the $(n,1)$-component of $B_1$.
	\end{itemize}
	These maps induce a homogeneous homomorphism $\varphi\colon k\widehat{Q}\to B$, which clearly preserves the relations. Denoting by $I$ the ideal generated by the relations, we obtain a homomorphism $k\widehat{Q}/I\to B$. Since it is clearly an isomorphism in degree $0$, it is enough to consider the degree $1$ part. Let $e_{(i,l)}$ be the idempotent of $k\widehat{Q}/I$ at the vertex $(i,l)$, and set $e_l=\sum_{i\in Q_0}e_{(i,l)}$. We denote their images under $\varphi$ by the same symbols. It is sufficient to show that $\varphi$ induces an isomorphism $e_l(k\widehat{Q}/I)e_m \to e_lBe_m$ for each $1\leq l, m\leq n$. By the relation (iii), each term is $0$ in degree $1$ unless $m-l=1$ or $(l,m)=(n,1)$, so we only have to consider these two cases.
	
	{\it Case 1: The map $e_l(k\widehat{Q}/I)e_{l+1} \to e_lBe_{l+1}$ is an isomorphism for each $1\leq l<n$.} By the relation (i), any element in $e_l(k\widehat{Q}/I)e_{l+1}$ can be written as $a\cdot(\sum_{i\in Q_0}v^l_i)$ for some $a\in kQ^l=A$. This observation immediately shows the map is an isomorphism.
	
	{\it Case 2: The map $e_n(k\widehat{Q}/I)e_{1} \to e_nBe_1$ is an isomorphism.} By the relation (ii), the space $e_n(k\widehat{Q}/I)e_{1}$ is isomorphic to the degree $1$ part of the preprojective algebra of $Q$, thus to $U$. On the other hand, the space $e_nBe_1$ is also clearly $U$.
\end{proof}
We look at the most special case $d=1$ and $n=1$.
\begin{Ex}\label{usual}
	Let $Q$ be a finite connected acyclic non-Dynkin quiver and $kQ$ its path algebra, which is $1$-representation-infinite. 
	
	The $1$-Iwanaga-Gorenstein algebra $B=kQ\oplus U$ with $U=\tau^{-1}kQ$ is a truncation of the preprojective algebra $\Pi$ of $Q$, which is presented by the same quiver as $\Pi$ and the additional relations `the elements of $U$ square to zero', as stated in Proposition \ref{Qhat}.
	
	The equivalence $\D_\sg(B)\simeq\C_2(kQ)$ in Corollary \ref{d=1} is given in \cite{BIRSc,ART} as the $2$-CY category associated to the square of the Coxeter element in the Coxeter group of $Q$. 
	Our proof is different from theirs since our equivalence comes from quasi-equivalence of DG orbit categories.
\end{Ex}

The next example is the case $d=1$ and $n=2$.
\begin{Ex}
	Let $Q$ be the following quiver
	\[ \xymatrix{
		\circ\ar@<0.4ex>[r]^a\ar@<-0.3ex>[r]_b& \circ\ar[r]^c&\circ}, \]
	thus $A=kQ$ is $1$-representation infinite. Let $n=2$, so by Proposition \ref{Qhat}, the $1$-Iwanaga-Gorenstein algebra ${B}$ is presented by the following quiver with relations.
	
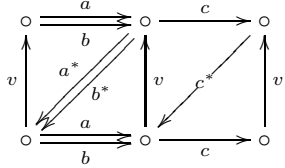
\[ \xymatrix@R=5mm@C=12mm{
		\circ\ar@<0.4ex>[r]^a\ar@<-0.3ex>[r]_b& \circ\ar[r]^c\ar@<0.4ex>[ddl]_{a^\ast}\ar@<-0.3ex>[ddl]^{b^\ast}&\circ\ar[ddl]|{c^\ast} \\ \\
		\circ\ar@<0.4ex>[r]^a\ar@<-0.3ex>[r]_b\ar[uu]^v& \circ\ar[r]_c\ar[uu]_v&\circ\ar[uu]_v }\qquad
	   \xymatrix@R=2mm{
	   	av=va,\, bv=vb,\,cv=vc\\
	   	a^\ast a+b^\ast b=0,\,aa^\ast+bb^\ast=c^\ast c,\,cc^\ast=0 \\
	   	va^\ast=0,\, vb^\ast=0,\, vc^\ast=0,\quad a^\ast v=0,\, b^\ast v=0,\, c^\ast v=0 } \]
	By Corollary \ref{d=1}, we have triangle equivalences $\C_3(kQ)\simeq\D_\sg({B})$.
\end{Ex}

\subsection{The case $n=1$.}
We now turn to another special case of $n=1$. In this case, the algebra $B$ is a truncation of the $(d+1)$-preprojective algebra of $A$.
\begin{Cor}\label{n=1}
	Let $A$ be a $d$-representation infinite algebra, $U=\Ext_A^d(DA,A)$, and $B=A\oplus U$. Then $B$ is $d$-Iwanaga-Gorenstein and there is a triangle equivalence $\C_{d+1}(A)\simeq\D_\sg(B)$.
\end{Cor}
This is a higher dimensional analogue of Example \ref{usual} above. It is predicted in \cite{Iydc} as a generalization of \cite{BIRSc} that $\D_\sg(B)$ has a $(d+1)$-cluster tilting object. We deduce this from our equivalence with the $(d+1)$-cluster category.

Let us now give an example. See also Example \ref{di}(1) for an example in $d=2$.
\begin{Ex}
(1)  Let $A$ be the tensor product of two path algebras of Kronecker quivers, thus is presented by the following quiver with relations.
\newcommand{\arvuv}{\ar@<0.3ex>[d]_u\ar@<-0.3ex>[d]^v}
\[ \xymatrix{1{\ar@<0.55ex>[r]_y\ar@<-0.05ex>[r]^x}\arvuv& 2\arvuv \\ 3\ar@<0.55ex>[r]_y\ar@<-0.05ex>[r]^x& 4 } \quad {xu=ux, \, xv=vx, \, yu=uy, \, yv=vy.} \]
This is a $2$-representation infinite algebra. (See \cite[Theorem 2.10]{HIO}.) This is also the endomorphism algebra of a tilting bundle $\mathcal{O}\oplus\mathcal{O}(1,0)\oplus\mathcal{O}(0,1)\oplus\mathcal{O}(1,1)$ over $\P^1\times\P^1$. The preprojective algebra $\Pi$ of $A$ is presented by the following quiver with suitable commutativity relations,
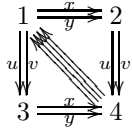
\[ \xymatrix{1\ar@<0.55ex>[r]_y\ar@<-0.05ex>[r]^x\arvuv& 2\arvuv \\ 3\ar@<0.55ex>[r]_y\ar@<-0.05ex>[r]^x& 4\ar@<0.8ex>[ul]\ar@<0.4ex>[ul]\ar@<0.0ex>[ul]\ar@<-0.4ex>[ul] } \]
thus so is its truncation $B$ with some additional relations.
By Corollary \ref{n=1} we have a triangle equivalence $\C_3(A)\simeq\D_\sg(B)$.

(2) Let $A'$ be the algebra presented by the following quiver with relations.
\[ \xymatrix{ 1\ar@<0.7ex>[r]^x\ar@<-0.2ex>[r]_y& 2\ar@<0.7ex>[r]^u\ar@<-0.2ex>[r]_v& 3\ar@<0.7ex>[r]^x\ar@<-0.2ex>[r]_y&4 }, \quad xuy=yux, \, xvy=yvx. \]
This is obtained from $A$ by taking for example the left mutation at vertex $3$, that is, the endomorphism algebra of another tilting bundle $\mathcal{O}\oplus\mathcal{O}(1,0)\oplus\mathcal{O}(1,1)\oplus\mathcal{O}(2,1)$. Then it is easy to see that $A'$ is also $2$-representation infinite. Its preprojective algebra $\Pi'$ is given by the following quiver with commutativity relations,
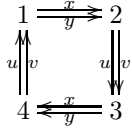
\[ \xymatrix{ 1\ar@<0.55ex>[r]_y\ar@<-0.05ex>[r]^x&2\arvuv\\4\ar@<0.3ex>[u]_v\ar@<-0.3ex>[u]^u&3\ar@<0.05ex>[l]_x\ar@<-0.55ex>[l]^y } \]
thus its truncation $B'$ by the same quiver with suitable additional relations.
By Corollary \ref{n=1} we have a triangle equivalence $\C_3(A')\simeq\D_\sg(B')$.

(3)  The $2$-representation infinite algebras $A$ and $A'$ above are derived equivalent, hence their cluster categories are equivalent; $\C_3(A)\simeq\C_3(A')$. Therefore we deduce that all the relevant $3$-CY categories are equivalent; $\D_\sg(B)\simeq\C_3(A)\simeq\C_3(A')\simeq\D_\sg(B')$.
\end{Ex}

\section{Examples: Polynomial rings}\label{poly}
In this section, we apply our main results for polynomial rings and give a concrete description of the $d$-representation-infinite algebra $A$ and the $d$-Iwanaga-Gorenstein algebra $B$. (See (\ref{eqAUB}) for the definition of $A$ and $B$.)

Let $R=k[x_0,\ldots,x_d]$ with $\deg x_i=-a_i<0$. It is bimodule $(d+1)$-CY algebra with $a$-invariant $a=\sum_{i=0}^da_i$. We have the following result on the algebras $A$ and $B$.
For a finite subgroup $G\subset\mathrm{GL}_{d+1}(k)$, which naturally acts on $R$, the {\it skew group algebra} $R\ast G$ is a vector space $R\otimes_kkG$ with multiplication $(a\otimes g)(b\otimes h)=ag(b)\otimes gh$.
\begin{Prop}\label{typeA}
Suppose that $a$ is invertible in $k$, there exists a primitive $a$-th root of unity $\zeta \in k$, and that $a_0,\ldots,a_d$ are relatively prime.
\begin{enumerate}
\item The algebra $A$ is presented by the quiver $Q$ with the vertices $\{0, 1,\ldots,a-1 \}$, the arrows $x_i=x_i^l \colon l \to l+a_i$ for each $0\leq i\leq d$ and $0\leq l\leq a-1$ such that $l+a_i\leq a-1$, and with the commutativity relations $x_j^{l+a_i}x_i^l=x_i^{l+a_j}x_j^l$.
\item Let $g=\mathrm{diag}(\zeta^{a_0},\ldots,\zeta^{a_d})\in\mathrm{SL}_{d+1}(k)$ and $G$ the cyclic subgroup generated by $g$. Then, the $(d+1)$-preprojective algebra of $A$ is isomorphic to $R \ast G$.
\item $A$ is $d$-representation-infinite of type $\widetilde{A}$, in the sense of \cite{HIO}.
\item $B$ is presented by the quiver $\widehat{Q}$ obtained by adding to $Q$ the arrows $u=u^l\colon l\to l-1$ for each $1\leq l\leq a-1$ and two types of additional relations:
\begin{enumerate}
\renewcommand{\labelenumii}{(\roman{enumii})}
\item $x_i^{l-1}u^l=u^{l+a_i}x_i^l$ whenever $1\leq l$ and $l+a_i\leq a-1$.
\item $u^{l+a_i-1}x^{l-1}_iu^l=0$ whenever $1\leq l$ and $l+a_i\leq a$.
\end{enumerate}
\end{enumerate}
\end{Prop}
\begin{proof}
(1)  Note that the category $\proj^\Z\!R$ is presented by the quiver with vertices set $\Z$, arrows $x_i^l \colon l \to l+a_i$, and with the commutativity relations. Since $A\simeq\End_R^\Z(R \oplus R(-1) \oplus \cdots R(-(a-1)))$, it is presented by its full subquiver with vertices $\{ 0, 1,\ldots,a-1 \}$ and with the induced relations.\\
(2)(3) We follow the construction of $d$-representation-infinite algebras of type $\widetilde{A}$ \cite[Section 5]{HIO}.
	
Let $L$ be the free $\Z$-module with basis $\a_1,\ldots,\a_d$, and set $\a_0:=-\a_1-\cdots-\a_d$. Let $Q$ be the quiver with vertices $Q_0=L$ and the set of arrows $Q_1=\{ x_i=x_i^l \colon l \to l+\a_i \mid l \in L, \,  0\leq i \leq d \}$. Moreover for each $l \in L$ and $0\leq i<j\leq d$, define the relation $r_{ij}=r_{ij}^l=x_ix_j-x_jx_i$. Then we have a category $\mathcal{L}$ presented by this quiver and relations. We assign for each point $l=\sum_{i=1}^dl_i\a_i \in L$ the integer $m(l)=\sum_{i=1}^dl_ia_i$. 
	
Now let $B \subset L$ the subgroup consisting of points $l$ in $L$ such that $m(l)$ is a multiple of $a$. The subgroup $B$ has finite index $a$, and acts on $\mathcal{L}$ by translation. We then have the orbit category $\mathcal{L}/B$, which can naturally be identified with the algebra $\Pi$ presented by the (finite) quiver $Q/B$ and the induced relations. By \cite[Lemma 5.3]{HIO}, $\Pi$ is isomorphic to the skew group algebra $R\ast G$.

Going back to the original quiver $Q$, we set
\[  C= \{ x \colon l \to l' \text{ in } Q_1 \mid m(l)<na\leq m(l') \text{ for some } n \in \Z \}, \]
which is a {periodic and bounding cut} in the sense of \cite[Definitions 5.4, 5.5]{HIO}, and is stable under $B$. Then $C$ induces a grading on $Q/B$, hence on $\Pi$ by
\[ \deg x=\begin{cases} 1 & (x \in C) \\ 0 & (x \not\in C) \end{cases} \]
for each $x \in Q_1$. Now by \cite[Theorem 5.6]{HIO}, $\Pi$ is the preprojective algebra of its degree $0$ part, and we see that it is nothing but $A$, since they are presented by the quiver $(Q/B)\setminus (C/B)$ and the commutativity relations.\\
(4)  We first compute the quiver of $B$ using Lemma \ref{comp}(3). Since $R$ is generated by degree $>-a$, the vector space on the lower left corner of $U/(J_AU+UJ_A)$ is $0$. Therefore the arrows we have to add are just the ones corresponding to $1\in R_0/J_0=k$, and the quiver of $B$ is $\widehat{Q}$. Then there exists a natural homomorphism $k\widehat{Q}\to B$, which preserves the relations, thus induces a homomorphism $\varphi\colon k\widehat{Q}/I\to B$, where $I$ is the ideal generated by the relations. We show that $\varphi$ is an isomorphism. 
We may truncate by the idempotents; denote by $e_i$ the idempotent of $k\widehat{Q}/I$ at vertex $i$, and we show that the induced map $e_j(k\widehat{Q}/I)e_i \to e_jBe_i$ is an isomorphism for each $0\leq i,j\leq a-1$. Note that the relations show that each space $e_j(k\widehat{Q}/I)e_i$ has a basis consisting of monomials of degree $-(j-i+1)$ (defined from the grading on $R$) each of which contains at most one of the $u^l$'s (which have degree $0$). It is then clear that $\varphi$ maps these basis to the basis of $e_jBe_i$.
\end{proof}

Let us first look at the easiest case. 
\begin{Ex}\label{2hensu}
This is a continuation of Example \ref{easiest}.
Let $R=k[x,y]$ with $\deg x=\deg y=-1$, so $R$ is $2$-CY  of $a$-invariant $2$. By Theorem \ref{dgcy}, $R^\dg$ is twisted $4$-CY. As we have seen in Example \ref{easiest}, we have an equivalence $\D^b(\qgr R) \simeq \D^b(\md A)$ with $A$ the Kronecker algebra, and its AR translation $\nu_1$ has a square root.
On the other hand, the Iwanaga-Gorenstein algebra $B$ is presented by the following quiver with relations:
\[ \xymatrix@C=12mm{ 0 \ar@<0.5ex>[r]^x\ar@<-0.4ex>[r]_y & 1\ar@/_15pt/@<-0.1ex>[l]_u}, \quad xuy=yux, \, uxu=uyu=0. \]
By (\ref{eqsumm}), there are equivalences
\[ \D^b(\md A)/\nu_1^{-1/2}[1] \simeq \D_\sg(B) \simeq \C(R^\dg) \]
of twisted $3$-CY categories; compare the case $m=2$ in Example \ref{kmv2}. Using the description as an orbit category, we can classify the objects in $\C(R^\dg)$ or $\D_\sg(B)$, which we leave to the reader.
\end{Ex}

We look at higher dimensional case.
\begin{Ex}
Let $R=k[x_0,x_1,\ldots,x_d]$ with $\deg x_0=\cdots=\deg x_d=-1$. Then $R$ is $(d+1)$-CY of $a$-invariant $d+1$, thus by Theorem \ref{dgcy}, $R^\dg$ is sign-twisted $(2d+2)$-CY. 
It is well-known that $\qgr R$ is equivalent to the category $\coh\P^d$ of coherent sheaves over the projective space $\P^d$. The tilting object in $\D^b(\qgr R)$ given in Proposition \ref{mm} is the tilting bundle $T=\bigoplus_{l=0}^d\mathcal{O}_{\P^d}(l)$ on $\P^d$, whose endomorphism ring $A$ is the $d$-Beilinson algebra. It is presented by the following quiver with commutativity relations:
\newcommand{\vd}{\rotatebox{90}{$\cdot\hspace{-1mm}\cdot\hspace{-1mm}\cdot$}}
\newcommand{\arrr}{\ar@<1.1ex>[r]\ar@<0.3ex>@{}[r]|\vd\ar@<-0.4ex>[r]}
\[ A=\xymatrix{ 0\arrr& 1\arrr&\cdots\arrr&d.} \]
The category $\add\{R(-i)\mid i\in\Z \}=\add\{\mathcal{O}(i)\mid i\in\Z\}=\add\{ \nu_d^{-i}A\mid i\in\Z\}$ (which are identified via the equivalence $\D^b(\qgr R)\simeq\D^b(\coh\P^d)\simeq\D^b(\md A)$) is presented by the following quiver.
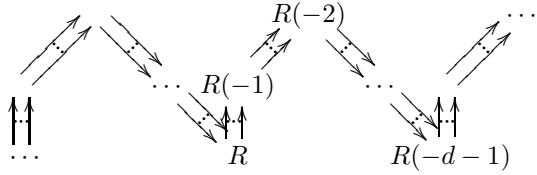
\newcommand{\hd}{\rotatebox{0}{$\cdot\hspace{-1mm}\cdot\hspace{-1mm}\cdot$}}
\newcommand{\nd}{\rotatebox{-45}{$\cdot\hspace{-1mm}\cdot\hspace{-1mm}\cdot$}}
\newcommand{\ld}{\rotatebox{45}{$\cdot\hspace{-1mm}\cdot\hspace{-1mm}\cdot$}}
\newcommand{\ara}{\ar@<1.0ex>[u]\ar@<0.3ex>@{}[u]|\hd\ar@<-0.4ex>[u]}
\newcommand{\arb}{\ar@<1.0ex>[ur]\ar@<0.3ex>@{}[ur]|\nd\ar@<-0.4ex>[ur]}
\newcommand{\arc}{\ar@<1.0ex>[dr]\ar@<0.3ex>@{}[dr]|\ld\ar@<-0.4ex>[dr]}
\[ \xymatrix@!R=2mm@!C=1mm{ &\arc&&&R(-2)\arc&&&\cdots\\ 
\arb&&\cdots\arc&R(-1)\arb&&\cdots\arc&\arb& \\
\cdots\ara&&&R\ara&&&R(-d-1)\ara& } \]
The autoequivalence $\nu_d^{-1/(d+1)}$ on $\D^b(\md A)$ acts on this subcategory by `moving one place' along the $d$-fold arrows. 
On the other hand the Iwanaga-Gorenstein algebra $B=A\oplus U$ is presented by the quiver
\[ \xymatrix@C=12mm{ 0\arrr & 1\arrr\ar@/_10pt/@<-0.1ex>[l]_u& \cdots\ar@/_10pt/@<-0.1ex>[l]_u\arrr& d\ar@/_10pt/@<-0.1ex>[l]_u} \]
and the commutativity relations, and $u^2=0$.
This is a truncation of $T_AU=A\oplus U\oplus U^2\oplus\cdots$, which is the endomorphism ring of a tilting object $\pi^\ast T$ over the total space of the line bundle $\pi\colon\mathcal{O}(-1)\to\P^d$.
Applying (\ref{eqsumm}), we have an embedding and an equivalence
\[ \D^b(\md A)/\nu_d^{-1/(d+1)}[1] \hookrightarrow \D_\sg(B) \simeq \C(R^\dg). \]
\end{Ex}


\section{Examples: Jacobian algebras arising from dimer models}\label{dimer}
A {\it dimer model} is a finite bipartite graph $\G$ on a real $2$-torus inducing a polygonal cell decomposition. We denote by $\G_0$, (resp. $\G_1$, $\G_2$) the set of vertices (resp. edges, faces) of $\G$. It gives rise to a quiver with potential $(Q,W)$ in the sense of \cite{DWZ} in the following way: Let $Q$ denote the dual quiver of $\G$, thus the set of vertices $Q_0$ (resp. arrows $Q_1$) corresponds bijectively to $\G_2$ (resp. $\G_1$). By convention, the arrows of $Q$ see white vertices on the right. Then for each vertex $v$ there is a unique cycle $c_v$ consisting of arrows of $Q$ corresponding to the edges of $\G$ which are adjacent to $v$. Now define the potential by $W=\sum_{v \colon \text{white}}c_v-\sum_{v \colon \text{black}}c_v$.

We assume that $\G$ is {\it consistent} (see \cite{Boc}) in the sense that there exists a map $\mathrm{R} \colon \G_1 \to \mathbb{R}_{>0}$ such that
\begin{itemize}
\item $\sum_{v\in\partial a}\mathrm{R}(a)=2$ for all $v\in\G_0$, where the sum runs over $a\in\G_1$ adjacent to $v$.
\item $\sum_{a\in\partial f}(1-\mathrm{R}(a))=2$ for all $f\in\G_2$, where the sum runs over $a\in\G_1$ in the boundary of $f$.
\end{itemize}
Fix a map
\begin{equation}\label{eqdeg}
d \colon Q_1=\G_1 \longrightarrow \Z
\end{equation}
such that $\sum_{v\in\partial a}d(a)$ is a constant $l$ for all $v\in\G_0$. 
Such maps are typically given by {perfect matchings} on $\G$. Recall that a {\it perfect matching} on a graph is a set of its edges such that each vertex is contained in precisely one edge in the set. It is known that the consistency condition ensures the existence of perfect matchings \cite[Section 2.3]{Br}. We can identify a perfect matching $P$ on $\G$ with a map $d\colon \G_1\to\{0,1\}$ such that $\sum_{v\in\partial a}d(a)=1$ for all $v\in\G_0$ by setting $d(a)=1$ if and only if $a\in P$. Consequently, any $\Z$-linear combination of perfect matchings gives a function (\ref{eqdeg}).

\begin{Prop}[See {\cite[Theorem 7.7]{Br}, \cite[Proposition 6.1]{AIR}}]
Let $\G$ be a consistent dimer model, and let $d$ be a map (\ref{eqdeg}) such that $\sum_{v\in\partial a}d(a)=l$ for all $v\in\G_0$. Then $d$ gives a grading on the Jacobian algebra making it into a bimodule $3$-CY algebra of $a$-invariant $-l$.
\end{Prop}
\begin{proof}
	Give a grading on the quiver $Q$ by setting $\deg a=d(a)$ for $a\in Q_1$. Then the potential $W$ is homogeneous of degree $l$, thus $d$ induces a grading on the Jacobian algebra $R$. Consider the complex
	\[ \xymatrix{\disoplus_{i\in Q_0}Re_i\otimes e_iR\ar[r]^-{d_3}& \disoplus_{a\in Q_1}Re_{s(a)}\otimes e_{t(a)}R\ar[r]^-{d_2}& \disoplus_{a\in Q_1}Re_{t(a)}\otimes e_{s(a)}R\ar[r]^-{d_1}& \disoplus_{i\in Q_0}Re_i\otimes e_iR } \]
	with maps
	\begin{equation*}
	\begin{aligned}
		d_1(e_{t(a)}\otimes e_{s(a)})&=a\otimes e_{s(a)}-e_{t(a)}\otimes a \\
		d_2(e_{s(a)}\otimes e_{t(a)})&=\sum_{b\in Q_1}p\otimes q \text{ for each cycle } apbq \text{ in } W \\
		d_3(e_i\otimes e_i)&=\sum_{t(a)=i}a\otimes e_i-\sum_{s(a)=i}e_i\otimes a.
	\end{aligned}
	\end{equation*}
	By \cite[Theorem 7.7]{Br}, this complex together with the multiplication map $\bigoplus_{i\in Q_0}Re_i\otimes e_iR \to R$ gives a bimodule projective resolution of $R$ making it into a $3$-CY algebra. Now, taking degree into account, we deduce that $R$ is graded bimodule $3$-CY of $a$-invariant $-l$.
\end{proof}

\begin{Ex}\label{di}
Let $\G$ be a dimer model as in the left picture below, where the vertical and horizontal ends are identified so that it has $4$ faces which are labeled by 1, 2, 3, and 4. It gives a $3$-CY algebra $R$ presented by the quiver in the right picture.
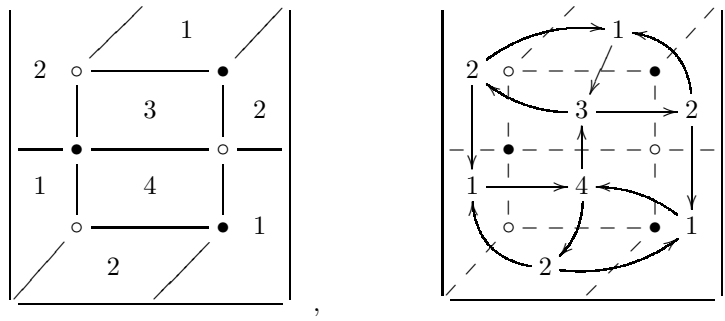
\[ \xymatrix@R=1mm@C=1mm{
	\ar@{-}[rrrrrrrr]\ar@{-}[dddddddd]&&&&&&&&\ar@{-}[dddddddd] \\
	&&&&& 1 &&& \\
	&2 & \circ\ar@{-}[uurr]\ar@{-}[rrrr]\ar@{-}[dd]&&&& \bullet\ar@{-}[uurr]\ar@{-}[dd]& & \\
	&&&& 3 &&&2& \\
	\ar@{-}[rr]&& \ar@{-}[rrrr]\ar@{-}[dd]\bullet&&&&\circ\ar@{-}[rr]\ar@{-}[dd]&& \\
	&1&&& 4 &&&& \\
	&& \circ\ar@{-}[rrrr]&&&& \bullet& 1 & \\
	&&& 2 &&&&& \\
	\ar@{-}[rrrrrrrr]\ar@{-}[uurr]&&&& \ar@{-}[uurr] &&&& &,} \qquad\qquad
   \xymatrix@R=1mm@C=1mm{
	\ar@{-}[rrrrrrrr]\ar@{-}[dddddddd]&&&&&&&&\ar@{-}[dddddddd] \\
	&&&&& 1 \ar[ddl]&&& \\
	&2\ar@/^7pt/[rrrru]\ar[ddd] & \circ\ar@{--}[uurr]\ar@{--}[rrrr]\ar@{--}[dd]&&&& \bullet\ar@{--}[uurr]\ar@{--}[dd]& & \\
	&&&& 3\ar@/^5pt/[lllu]\ar[rrr] &&&2\ar@/_10pt/[uull]\ar[ddd]& \\
	\ar@{--}[rr]&& \ar@{--}[rrrr]\ar@{--}[dd]\bullet&&&&\circ\ar@{--}[rr]\ar@{--}[dd]&& \\
	&1\ar[rrr]&&& 4\ar[uu]\ar@/^5pt/[ddl] &&&& \\
	&& \circ\ar@{--}[rrrr]&&&& \bullet& 1\ar@/_5pt/[lllu] & \\
	&&& 2\ar@/^10pt/[uull]\ar@/_7pt/[rrrru] &&&&& \\
	\ar@{-}[rrrrrrrr]\ar@{--}[uurr]&&&& \ar@{--}[uurr] &&&& } \]
Now we consider the grading $d$ in (\ref{eqdeg}). We discuss two variations.

(1)  First we consider the grading below. The labels on the edges show the values under $d$, and unlabeled ones have degree $0$, thus it is a perfect matching. This grading makes $R$ into a bimodule $3$-CY algebra of $a$-invariant $1$, thus $A=R_0$ and $B=R_{\geq-1}$ are given by quivers below.
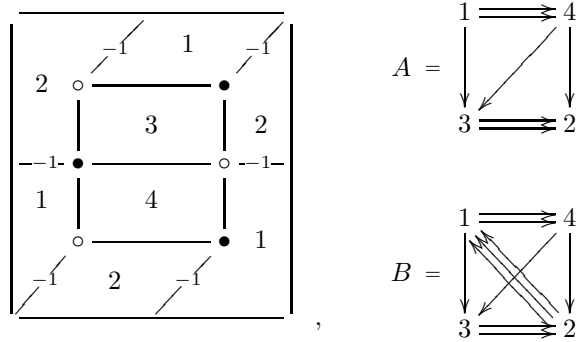
\newcommand{\arrarr}{\ar@<0.4ex>[rr]\ar@<-0.3ex>[rr]}
\newcommand{\triul}{\ar@<1.0ex>[uull]\ar@<0.3ex>[uull]\ar@<-0.4ex>[uull]}
\[ \xymatrix@R=1mm@C=1mm{
	\ar@{-}[rrrrrrrr]\ar@{-}[dddddddd]&&&&&&&&\ar@{-}[dddddddd] \\
	&&&&& 1 &&& \\
	&2 & \circ\ar@{-}[uurr]|{-1}\ar@{-}[rrrr]\ar@{-}[dd]&&&& \bullet\ar@{-}[uurr]|{-1}\ar@{-}[dd]& & \\
	&&&& 3 &&&2& \\
	\ar@{-}[rr]|{-1}&& \ar@{-}[rrrr]\ar@{-}[dd]\bullet&&&&\circ\ar@{-}[rr]|{-1}\ar@{-}[dd]&& \\
	&1&&& 4 &&&& \\
	&& \circ\ar@{-}[rrrr]&&&& \bullet& 1 & \\
	&&& 2 &&&&& \\
	\ar@{-}[rrrrrrrr]\ar@{-}[uurr]|{-1}&&&& \ar@{-}[uurr]|{-1} &&&& &,} \qquad
  \xymatrix@R=3mm@C=4mm{
  	&1\arrarr\ar[dd]&& 4\ar[dd]\ar[ddll]\\ 
  	A\ar@{}|{=}[r]&&&\\
  	&3\arrarr&& 2 \\
  	&&& \\
  	&1\arrarr\ar[dd]&& 4\ar[dd]\ar[ddll]\\ 
  	B\ar@{}|{=}[r]&&&\\
  	&3\arrarr&& 2\triul\ar@{}[r]_<<{.}& } \]
This $2$-representation infinite algebra $A$ is the endomorphism ring of a tilting bundle $T$ on the Hirzeburch surface $\Sigma_1=\mathbb{P}(\mathcal{O}_{\mathbb{P}^1}\oplus\mathcal{O}_{\mathbb{P}^1}(-1))$ which is the blow-up of $\mathbb{P}^2$ at one point as well. Also, $R$ is the $3$-preprojective algebra of $A$, which is the endomorphism ring of a tilting bundle $\pi^\ast T$ on the total space of the canonical bundle $\pi\colon\omega\to\Sigma_1$ over $\Sigma_1$ (see \cite{BuH}).
Now the DG algebra $R^\dg$ is $4$-CY by Theorem \ref{dgcy}, and applying Corollary \ref{n=1} we have triangle equivalences
\[ \D_\sg(B)\simeq\C_3(A)=\C(R^\dg). \]

(2)  We next consider the grading below, again the non-zero degree of each edge is labeled. This is not given by a perfect matching (but by a sum of two perfect matchings), and makes $R$ into a bimodule $3$-CY algebra of $a$-invariant $2$. In this case, the $2$-representation infinite algebra $A=\begin{pmatrix}R_0& 0 \\ R_{-1}& R_0\end{pmatrix}$ and the $2$-Iwanaga-Gorenstein algebra $B=A\oplus U$ with $U=\begin{pmatrix}R_{-1}& R_0 \\ R_{-2}& R_{-1}\end{pmatrix}$ is presented as follows.
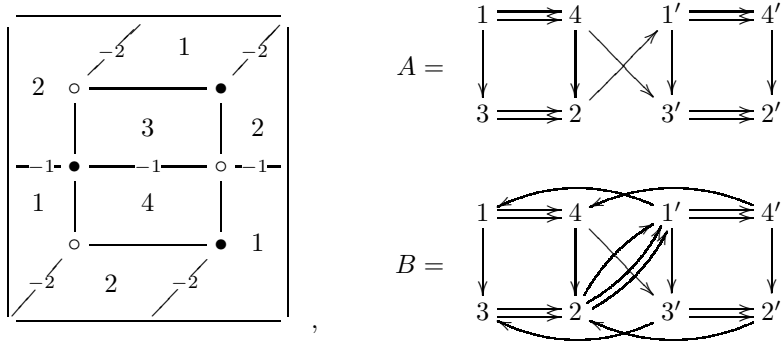
\newcommand{\arr}{\ar@<0.4ex>[r]\ar@<-0.3ex>[r]}
\newcommand{\ardr}{\ar@<0.4ex>@/_5pt/[ur]\ar@<-0.3ex>@/_5pt/[ur]}
  \[ \xymatrix@R=1mm@C=1mm{
	\ar@{-}[rrrrrrrr]\ar@{-}[dddddddd]&&&&&&&&\ar@{-}[dddddddd] \\
	&&&&& 1 &&& \\
	&2 & \circ\ar@{-}[uurr]|{-2}\ar@{-}[rrrr]\ar@{-}[dd]&&&& \bullet\ar@{-}[uurr]|{-2}\ar@{-}[dd]& & \\
	&&&& 3 &&&2& \\
	\ar@{-}[rr]|{-1}&& \ar@{-}[rrrr]|{-1}\ar@{-}[dd]\bullet&&&&\circ\ar@{-}[rr]|{-1}\ar@{-}[dd]&& \\
	&1&&& 4 &&&& \\
	&& \circ\ar@{-}[rrrr]&&&& \bullet& 1 & \\
	&&& 2 &&&&& \\
	\ar@{-}[rrrrrrrr]\ar@{-}[uurr]|{-2}&&&& \ar@{-}[uurr]|{-2} &&&& &,} \qquad
   \xymatrix{ 
   	1\arr\ar[d]_*+{A=\hspace{3mm}}& 4\ar[dr]\ar[d]& 1'\ar[d]\arr& 4'\ar[d] \\
	3\arr& 2\ar[ur]& 3'\arr& 2' \\
	1\arr\ar[d]_*+{B=\hspace{3mm}}& 4\ar[dr]\ar[d]& 1\ar[d]'\arr\ar@/_10pt/[ll]& 4'\ar[d]\ar@/_10pt/[ll] \\
   	3\arr& 2\ar@/^5pt/[ur]\ardr& 3'\arr\ar@/^10pt/[ll]& 2'\ar@/^10pt/[ll] } \]
Applying our results, the DG algebra $R^\dg$ is sign twisted $5$-CY, and there exist an embedding and a triangle equivalence 
\[ \D^b(\md A)/\nu_2^{-1/2}[1] \hookrightarrow\D_\sg(B) \simeq \C(R^\dg). \]
\end{Ex}

\begin{appendix}
\renewcommand{\theequation}{\Alph{section}.\arabic{equation}}
\section{Multiplying gradings}\label{nbai}
Let $R$ be a graded ring. For a fixed integer $n\geq1$, define the graded ring ${}^n\!R$ by
\[	({}^n\!R)_i=\begin{cases} R_{i/n} & \text{if } n\mathrel{|}i \\0 & \text{if } n\!\mathrel{\not|}i\end{cases}. \]
If $R$ is bimodule $(d+1)$-CY of $a$-invariant $a$, then clearly ${}^n\!R$ is bimodule $(d+1)$-CY of $a$-invariant $na$. Although the category $\qper {}^n\!R$ just splits as a direct product of $n$ copies of $\qper R$ and yields nothing new, the cluster category $\C({}^n\!R^\dg)$, being a triangulated hull of $\qper {}^n\!R/(-1)[1]$, becomes `connected' by the action of the automorphism $(-1)[1]$, which yields something new.

The aim of this section is to describe the category $\C({}^n\!R^\dg)$ in terms of the relevant objects from $R$. Although this can be regarded as a special case of our main results, we shall obtain a better presentation of orbit categories.

Recall that $R$ is a negatively graded bimodule $(d+1)$-CY algebra of $a$-invariant $a$ such that each $R_i$ is finite dimensional, and recall the definitions of the $d$-representation infinite algebra $A=A(R)$, the cotilting bimodule $U=U(R)$, and the $d$-Iwanaga-Gorenstein algebra $B=B(R)$ from (\ref{eqAUB}).
We have the following description of $\C({}^n\!R^\dg)$ in terms of $A$, which generalizes Theorem \ref{reasonable} and Corollary \ref{maincor}.
\begin{Thm}\label{n}
There exists a fully faithful functor
\[ \xymatrix{\D^b(\md A)/\nu_d^{-1/a}[n]=\qper R/(-1)[n]\ar[r]&\C({}^n\!R^\dg) } \]
whose image generates $\C({}^n\!R^\dg)$ as a thick subcategory.
\end{Thm}
Let the $d$-representation infinite algebra $\widetilde{A}=A({}^n\!R)$, the cotilting $(\widetilde{A},\widetilde{A})$-bimodule $\widetilde{U}=U({}^n\!R)$, and the $d$-Iwanaga-Gorenstein algebra $\widetilde{B}=B({}^n\!R)$ be as given in (\ref{eqAUB}) for ${}^n\!R$, thus we have
\begin{equation}\label{hatAU}
	\widetilde{A}=
	\left( 
	\begin{array}{cccc}
		A & 0&\cdots& 0 \\
		0&A&\cdots&0\\
		\vdots& \vdots&\ddots&\vdots \\
		0&0&\cdots&A
	\end{array}     
	\right)=A\times\cdots\times A, \quad
	\widetilde{U}=\left( 
	\begin{array}{cccc}
		0 & A&\cdots& 0 \\
		\vdots& \vdots&\ddots&\vdots \\
		0&0&\cdots&A\\
		U&0&\cdots&0
	\end{array}     
	\right), \quad
	\widetilde{B}=\widetilde{A}\oplus\widetilde{U}.
\end{equation}
Then we also have the consequence in terms of singularity category of $\widetilde{B}$.
\begin{Cor}\label{Bhat}
	There exists a fully faithful functor
	\[ \xymatrix{ \D^b(\md A)/\nu_d^{-1/a}[n]\ar[r]& \D_\sg(\widetilde{B}) } \]
	whose image generates $\D_\sg(\widetilde{B})$ as a thick subcategory.
\end{Cor}

Now we start the proof. The first step is to apply our results in Section \ref{thms} for the CY algebra ${}^n\!R$. By Corollary \ref{root}, the $d$-AR translation $\nu_d$ on $\D^b(\md\widetilde{A})$ has an $na$-th root, and by Theorem \ref{reasonable} and Corollary \ref{maincor}, we have an equivalence and an embedding
\begin{equation}\label{eqn}
	\D^b(\md\widetilde{A})/\nu_d^{-1/na}[1]\simeq\qper {}^n\!R/(-1)[1]\hookrightarrow\C({}^n\!R^\dg).
\end{equation}

We next compare the derived orbit categories of ${}^n\!R$ (resp. $\widetilde{A}$) and of $R$ (resp. $A$). Obviously there is a diagram of equivalences and compatible autoequivalences
\[ \xymatrix@R=5mm@C=0.1mm{
	\ar@(ul,dl)[]_{(-1)}&\qper {}^n\!R\ar[d]_{\rotatebox{90}{$\simeq$}}\ar@{=}[rr]&& \qper R\times\cdots\times\qper R\ar[d]^{\rotatebox{-90}{$\simeq$}} \\
	\ar@(ul,dl)[]_{\nu_d^{-1/na}}&\D^b(\md\widetilde{A})\ar@{=}[rr]&& \D^b(\md A)\times\cdots\times\D^b(\md A). } \]
We describe the action of these autoequivalences on the right-hand-side.
\begin{Lem}\label{aut}
	\begin{enumerate}
		\item The action of $(-1)$ on $\qper {}^n\!R$ becomes $(X_1,\ldots,X_n)\mapsto(X_n(-1),X_1,\ldots,X_{n-1})$ on $\qper R\times\cdots\times\qper R$.
		\item The action of $\nu_d^{-1/na}$ on $\D^b(\md\widetilde{A})$ is $(X_1,\ldots,X_n)\mapsto(\nu_d^{-1/a}X_n,X_1,\ldots,X_{n-1})$ on $\D^b(\md A)\times\cdots\times\D^b(\md A)$.
	\end{enumerate}
\end{Lem}
\begin{proof}
	We only prove (2), the proof of (1) is similar. By the form (\ref{hatAU}) of $\widetilde{U}$, we see that $-\otimes^L_{\widetilde{A}}\widetilde{U}$ maps $(X_1,\ldots,X_n)$ to $(X_n\otimes^L_AU,X_1,\ldots,X_{n-1})$.
\end{proof}

We next relate the orbit categories arising from $R$ and ${}^n\!R$.
\begin{Lem}\label{orb}
	\begin{enumerate}
		\item The functor $\qper R\to\qper {}^n\!R$ given by $X\mapsto(X,0,\ldots,0)$ induces an equivalence
		\[ \qper R/(-1)[n]\xrightarrow{\simeq}\qper {}^n\!R/(-1)[1]. \]
		\item The functor $\D^b(\md A)\to\D^b(\md\widetilde{A})$ given by $X\mapsto(X,0,\ldots,0)$ induces an equivalence
		\[ \D^b(\md A)/\nu_d^{-1/a}[n]\xrightarrow{\simeq}\D^b(\md\widetilde{A})/\nu_d^{-1/na}[1]. \]
	\end{enumerate}
\end{Lem}
\begin{proof}
	Again, we only prove (2). We first want to show that there is a natural isomorphism
	\[ \xymatrix{ \disoplus_{l\in\Z}\Hom_{\D(A)}(X,(\nu_d^{-1/a}[n])^lY) \ar[r]& \disoplus_{l\in\Z}\Hom_{\D(\widetilde{A})}(FX,(\nu_d^{-1/na}[1])^lFY) }, \]
	where $FX=(X,0,\ldots,0)$. Since $F(\nu_d^{-1/a}Y[n])=(\nu_d^{-1/na}[1])^nF(Y)$ and $\Hom_{\D(\widetilde{A})}(FX,(\nu_d^{-1/na}[1])^lFY)=0$ unless $n\mathrel{|}l$ by Lemma \ref{aut}, we have a natural bijection.\\
	We next verify that the functor is dense. Note that $(0,\ldots,X_i,\ldots,0)\simeq((\nu_d^{-1/a}[1])^{-i+1}X_i,0\ldots,0)$ in the orbit category $\D^b(\widetilde{A})/\nu_d^{-1/na}[1]$. Therefore $\bigoplus_{i=1}^n(\nu_d^{-1/a}[1])^{-i+1}X_i\in\D^b(A)$ is mapped to $(X_1,\ldots,X_n)$.
\end{proof}
We now have our desired results.
\begin{proof}[Proof of Theorem \ref{n} and Corollary \ref{Bhat}]
	We have Theorem \ref{n} by (\ref{eqn}) and Lemma \ref{orb}. Then Corollary \ref{Bhat} follows by Theorem \ref{hope}.
\end{proof}

Let us demonstrate the difference of $\C(R^\dg)$ and $\C({}^n\!R^\dg)$. 
\begin{Ex}\label{ntrivial}
This is a generalization of Example \ref{trivial}, which is still almost trivial. Let
\[ R=k[x], \quad \deg x=-1, \]
which is bimodule $1$-CY of $a$-invariant $1$. Then we have $A=k$ and $U=k$. Now fix $n\geq1$ and consider the graded algebra ${}^n\!R$. We have
\[ \widetilde{A}=\begin{pmatrix}k&0&\cdots&0\\0&k&\cdots&0\\\vdots&\vdots&\ddots&\vdots\\0&0&\cdots&k\end{pmatrix}, \quad \widetilde{U}=\begin{pmatrix}0&k&\cdots&0\\\vdots&\vdots&\ddots&\vdots\\0&0&\cdots&k\\k&0&\cdots&0\end{pmatrix}, \]
thus $\widetilde{B}=\widetilde{A}\oplus\widetilde{U}$ is the self-injective Nakayama algebra with $n$ vertices and of Loewy length $2$. By Theorem \ref{n} and Corollary \ref{Bhat} we have equivalences of triangulated categories
\[ \D^b(\md k)/[n]\simeq\C({}^n\!R^\dg)\simeq\D_\sg(\widetilde{B}), \]
which is the $n$-cluster category of $k$.
\end{Ex}

\begin{Ex}\label{kmv4}
This is a generalization of Example \ref{kmv3}. As in Example \ref{kmv1}, \ref{kmv2}, and \ref{kmv3}, let
\[ R=k\!\left\langle x_1,\ldots,x_m\right\rangle /(x_1^2+\cdots+x_m^2), \quad \deg x_i=-1, \]
which is twisted $2$-CY of $a$-invariant $2$, the DG algebra $R^\dg$ is $4$-CY, and the $1$-representation infinite algebra $A$ is the path algebra $kQ_m$ of the $m$-Kronecker quiver.

Now we consider the cluster category $\C({}^n\!R^\dg)$. The algebra $\widetilde{A}$ is just the $n$ copies of $A=kQ_m$, and the $1$-Iwanaga-Gorenstein algebra $\widetilde{B}$ is presented by the following quiver with relations. 
\[ \xymatrix{
	\circ\ar@3[d]_m&\circ\ar[l]_v\ar@3[d]_m&\cdots\ar[l]_v&\circ\ar[l]_v\ar@3[d]^m\\
	\circ\ar[rrru]|u&\circ\ar[l]^v&\cdots\ar[l]^v&\circ\ar[l]^v\ar@{}[r]_,& }\quad
   \xymatrix@R=5mm{
   	x_iv=vx_i \, (1\leq i\leq m)\\ \sum_{i=1}^mx_iux_i=0, } \]
where we have denoted by $x_1,\ldots,x_m$ the $m$-fold arrows. By Theorem \ref{n}, we obtain triangle equivalences
\[ \D^b(\md kQ_m)/\nu_1^{-1/2}[n]\simeq\D_\sg(\widetilde{B})\simeq\C({}^n\!R^\dg). \]
Similarly to Corollary \ref{kmv3} and Remark \ref{rem1}, these are precisely the $(2n+1)$-CY triangulated category in \cite[Remark 3.4.5]{KMV}.
\end{Ex}

\section{$t$-structure in $\D^b(\md^\Z\!R)$}\label{D}
We give a version of Theorem \ref{tstr} for the derived category $\D^b(\md^\Z\!R)$ for graded coherent rings, as announced in Remark \ref{coh}. Let $R$ be a negatively graded, graded coherent ring. Then the category $\md^\Z\!R$ of finitely presented graded $R$-modules is abelian. We impose the following technical assumption.
\begin{itemize}
	\item[(R3)] The ideal $R_{>i}$ is finitely presented as a right $R$-module for each $i\leq0$.
\end{itemize}
Note that this is automatic when $R$ is Noetherian.
\begin{Lem}\label{trunc}
	Let $R$ be a negatively graded ring satisfying (R3) and let $X$ be a finitely presented graded $R$-module. Then the truncation $X_{>i}$ is finitely presented for each $i\in\Z$.
\end{Lem}
\begin{proof}
	Let $P_1\to P_0\to X\to0$ be a finite presentation of $X$ and consider its truncation $(-)_{>i}$. Since the $(P_0)_{>i}$ and $(P_1)_{>i}$ are finitely presented by the assumption (R3), so is $X_{>i}$.
\end{proof}

\begin{Thm}[cf. Theorem \ref{tstr}]\label{t}
	Let $R$ be a negatively graded, graded coherent ring satisfying (R3). Set
	\begin{equation*}
		\begin{aligned}
			t^{\leq0}&=\{ X \in \D^b(\md^\Z\!R) \mid H^i(X) \in \md^{\leq-i}\!R \text{ for all } i \in \Z \} ,\\
			t^{\geq0}&=\{ X \in \D^b(\md^\Z\!R) \mid H^i(X) \in \md^{\geq-i}\!R \text{ for all } i \in \Z \}.
		\end{aligned}
	\end{equation*}
	Then $(t^{\leq0},t^{\geq0})$ is a $t$-structure in $\D^b(\md^\Z\!R)$.
\end{Thm}
We give two independent proofs. The first one is a short proof using silting theory and DG categories. For the sake of reader who is not familiar with these, we include the second direct proof.
\subsection{The first proof}
Recall that we have a $t$-structure in the big derived category $\D:=\D(\Md^\Z\!R)$ which is given by 
\begin{equation*}
	\begin{aligned}
		\D_\M^{\leq0}&=\{ X \in \D(\Md^\Z\!R) \mid H^i(X) \in \Md^{\leq-i}\!R \text{ for all } i \in \Z \} ,\\
		\D_\M^{\geq0}&=\{ X \in \D(\Md^\Z\!R) \mid H^i(X) \in \Md^{\geq-i}\!R \text{ for all } i \in \Z \}.
	\end{aligned}
\end{equation*}
As in Section \ref{per}, we show that the $t$-structure $(\D_\M^{\leq0},\D_\M^{\geq0})$ above on $\D$ restricts to that on $\D^b(\md^\Z\!R)$.
\begin{proof}[Proof of Theorem \ref{t}]
	Since $R$ is right graded coherent the small derived category $\D^b(\md^\Z\!R)$ identifies with the thick subcategory of $\D$ whose cohomology is bounded and each one is finitely presented. Let $X\in\D^b(\md^\Z\!R)$ and consider the truncation triangle $X'\to X\to X'' \to X'[1]$ in $\D$. Since $X$ has bounded cohomology, so do $X'$ and $X''$ by Lemma \ref{ses}(1). Moreover, since each $H^iX$ is finitely presented, so are $H^iX'$ and $H^iX''$ by Lemma \ref{ses}(2) and Lemma \ref{trunc}. Therefore the $t$-structure in the big derived category restricts to that of the small one, which is precisely $(t^{\leq0},t^{\geq0})$.
\end{proof}

\subsection{The second proof}
We turn to the second direct proof. In this subsection we will use $\D$ for the small derived category $\D^b(\md^\Z\!R)$. We need several lemmas for the proof. Put, as usual, $t^{\leq n}=t^{\leq 0}[-n]$ and $t^{\geq n}=t^{\geq 0}[-n]$. The first one is obvious.
\begin{Prop}
	We have $t^{\leq-1} \subset t^{\leq0}$ and $t^{\geq1} \subset t^{\geq0}$.
\end{Prop}

The following easy observations will be useful.
\begin{Lem}\label{easy}
	Let $\A$ be an abelian category with enough projectives $\mathcal{P}$. Let $P \in \K^-(\mathcal{P})$, $X \in \D^b(\A)$ and suppose that $\Hom_\A(P^i, H^i(X))=0$ for all $i \in \Z$. Then $\Hom_{\D(\A)}(P,X)=0$.
\end{Lem}
\begin{proof}
	We may assume by induction on the length of $X$ that $X \in \A$. Then we have $\Hom_{\D(\A)}(P,X)=\Hom_{\K(\A)}(P,X) \twoheadleftarrow \Hom_{\C(\A)}(P,X) \subset \Hom_\A(P^0,X)=0$.
\end{proof}

\begin{Lem}\label{resol}
	Let $X \in t^{\leq0}$. Then there exists a projective resolution $P \to X$ such that each term $P^i \in \add R(\geq i)$ for each $i \in \Z$.
\end{Lem}
\begin{proof}
	This can be seen by recalling the construction of a projective resolution: Suppose we have constructed such $P$ for degree $\geq n$. As in the diagram below, let $B=\Ker(P^n \to C)$, $A$ the pull-back, and $P^{n-1} \to A$ a surjection from a projective. 
	\newdir{ >}{{}*!/-5pt/@{>}}
	\[ \xymatrix@C=3mm@R=3mm{
		P^{n-1}\ar@{-->}[rrr]\ar@{-->}[ddd]\ar@{-->>}[dr]&&& P^n\ar[rr]\ar[ddd]\ar@{->>}[dr] && \\
		& A\ar@{-->}[ddd]\ar@{-->>}[r]\ar@{}[dddr]|{\mathrm{PB}}& B\ar@{ >-->}[ur]\ar@{-->}[ddd]&& C\ar[ddd]\ar@{->>}[r]& \\
		\\
		X^{n-1}\ar@{->>}[dr]\ar[rrr]&&& X^n\ar[rr]\ar@{->>}[dr] &&\\
		& C^{n-1}\ar@{->>}[r]& B^n\ar@{ >->}[ur] && C^n \ar@{->>}[r]& } \]
	Then $B \in \md^{\leq-n}R$ since it is a subset of $P^n$ and $P^n \in \add R(\geq\!n)$. Also, since there exists an exact sequence $0 \to H^{n-1}(X) \to A \to B \to 0$ and $H^{n-1}(X) \in \md^{\leq-n+1} R$ by $X \in t^{\leq0}$, we have $A \in \md^{\leq-n+1}$. Therefore we can take its projective cover $P^{n-1} \in \add R(\geq n-1)$.
\end{proof}

These observations yield the following.
\begin{Prop}
	We have $\Hom_\D(X,Y)=0$ for all $X \in t^{\leq0}$ and $Y \in t^{\geq1}$.
\end{Prop}
\begin{proof}
	Take a projective resolution $P \to X$ in Lemma \ref{resol}. On the other hand, we have $H^i(Y) \in \md^{>-i}R$ for each $i \in \Z$. Therefore we deduce $\Hom_\D(X,Y)=\Hom_\D(P,Y)=0$ by Lemma \ref{easy}.
\end{proof}

We now give a final observation.
\begin{Prop}
	For any $X \in \D$, there exists a triangle $X' \to X \to X'' \to X'[1]$ in $\D$ with $X' \in t^{\leq0}$ and $X'' \in t^{\geq 1}$.
\end{Prop}
\begin{proof}
	We proceed by induction on $w(X)=\max\{i \in \Z\mid H^i(X)\neq0\}-\min\{i \in \Z\mid H^i(X)\neq0\}$.
	
	If $w(X)=0$, then $X \in (\md^\Z\!R)[n]$ for some $n \in \Z$. In this case, truncating the graded module $X$ as $0 \to X_{\leq -n} \to X \to X_{>-n} \to 0$ in $\md^\Z\!R$ yields a desired triangle by Lemma \ref{trunc}.
	
	If $w(X)>0$, there exists $n \in \Z$ such that in the truncation $X^{\leq n} \to X \to X^{>n} \to X^{\leq n}[1]$ of $X$ with respect to (the shift of) the standard $t$-structure $(\D^{\leq n}, \D^{\geq n})$, one has $w(Y),w(Z)<w(X)$, where $Y:=X^{\leq n}$ and $Z:=X^{>n}$. By induction hypothesis, there exist triangles $Y' \to Y \to Y'' \to Y'[1]$ and $Z' \to Z \to Z'' \to Z'[1]$ such that $Y', Z' \in t^{\leq0}$ and $Y'', Z'' \in t^{\geq1}$, thus the diagram below.
	\[ \xymatrix{
		Z'[-1]\ar[r]\ar@{-->}[d]& Z[-1]\ar[r]\ar[d]& Z''[-1]\ar[r]\ar@{-->}[d]& Z'\ar@{-->}[d] \\
		Y'\ar[r]& Y\ar[r]\ar[d]& Y''\ar[r]& Y'[1] \\
		& X\ar[d]& & \\
		Z'\ar[r]& Z\ar[r]& Z''\ar[r]& Z'[1] } \]
	
	We claim that $\Hom_\D(Z'[-1],Y'')=0$. This allows us to complete the morphism $Z[-1] \to Y$ to a morphism of triangles as in the dashed line above, thus the diagram above to a $3 \times 3$ diagram of triangles by \cite[Proposition 1.1.11]{BBD}. We then have a triangle $X' \to X \to X'' \to X'[1]$ in the third row, which is a desired one since the first and the third column are triangles and $t^{\leq0}$ and $t^{\geq0}$ are extension-closed.
	
	We now prove the claim. Observe that any triangle $W' \to W \to W'' \to W'[1]$ in $\D$ with $W' \in t^{\leq0}$ and $W'' \in t^{\geq1}$ yields a short exact sequence
	\[ \xymatrix{ 0 \ar[r]& H^i(W')\ar[r]& H^i(W)\ar[r]& H^i(W'')\ar[r]& 0} \]
	in $\md^\Z\!R$ for all $i \in \Z$. Indeed, we have $H^{i-1}(W'')\in\md^{>-i+1}R$ by $W'' \in t^{\geq1}$ and $H^i(W') \in \md^{\leq-i}R$ by $W' \in t^{\leq0}$, thus the connecting homomorphisms are $0$. In particular, $W', W'' \in \D^{\leq n}$ (resp. $\in \D^{\geq n}$) if and only if $W \in \D^{\leq n}$ (resp. $\in \D^{\geq n}$).
	
	Now apply the above argument to $Z' \to Z \to Z'' \to Z'[1]$, which shows $Z' \in \D^{\geq n+1}$, hence $Z'[-1] \in t^{\leq1}\cap\D^{\geq n+2}$. Therefore $H^i(Z'[-1])\in \md^{\leq -n-1} R$ for all $i \in \Z$ since it is $0$ for $i\leq n+1$ and is in $\md^{\leq-i+1}R$ for $i \geq n+2$.
	Similarly, we have $H^i(Y'') \in \md^{\geq-n+1}R$ for all $i$ since $Y'' \in \D^{\leq n}\cap t^{\geq1}$.
	We therefore obtain the claim.
\end{proof}
Now Theorem \ref{t} is a consequence of the Propositions.

\section{Proof of Proposition \ref{mm}}\label{mmpr}
In this section we give a proof of Minamoto--Mori's equivalence \cite{MM} based on Theorem \ref{dact}. The main tool is the realization of the Verdier quotient as a subcategory given in Theorem \ref{dact}(2).

We need the following computation of morphism in $\qper R$ which is not covered by Theorem \ref{dact}(2).
\begin{Lem}\label{nex}
	Let $X, Y \in \Md^\Z\!R$ such that $X, Y\in\per^\Z\!R$ and $\Hom_R^\Z(L,Y)=0$ for all $L\in\fl^\Z\!R$. Then we have $\Hom_{\qper R}(X,Y[<\!0])=0$.
\end{Lem}
\begin{proof}
	Let $l<0$ and let a morphism $X\to Y[l]$ in $\qper R$ be presented by a diagram $X\xleftarrow{s}Z\xrightarrow{}Y[l]$ in $\per R$ with $\cone s\in\D^b(\fl^\Z\!R)$. We claim that we can replace $s$ by a morphism whose cone lies in $\fl^\Z\!R$.
	\[ \xymatrix@!R=2mm@!C=2mm{
		&&Z'\ar@{-->}[dl]_-{s'}\ar[dr]&&\\
		&X\ar@{-->}[dl]\ar@{=}[dr]&&Z\ar[dl]_-s\ar[dr]&\\
		L^{\leq0}\ar[dr]&&X\ar[dl]&&Y[l]\\
		&L&&&}\qquad\qquad
	\xymatrix@!R=2mm@!C=2mm{
		&&Z'\ar@{=}[dr]\ar[dl]&&\\
		&Z''\ar[dr]_-{s''}\ar[dl]\ar@{-->}@/^/[drrr]&&Z'\ar[dl]^-{s'}\ar[dr]&\\
		M^{<0}\ar[dr]&&X\ar[dr]\ar[dl]&&Y[l]\\
		&M\ar[dr]&&H^0M\ar@{=}[dl]&\\
		&&H^0M&& } \]
	
	First complete $s$ to a triangle $Z\to X\xrightarrow{s}L\to Z[1]$ and consider the truncation $L^{\leq0}\to L\to L^{>0}\to L^{\leq0}[1]$ with respect to the standart $t$-structure. Since $\Hom_{\per^\Z\!R}(X,L^{>0})=0$ there is a map $X\to L^{\leq0}$ and as in the left diagram above, the original moprhism equals the morphism $X\xleftarrow{s'}Z'\to Z\to Y[l]$ with $\cone s'=L^{\leq0}$ concentrated in (cohomological) degree $\leq0$.
	
	Next consider the truncation of $M:=L^{\leq0}$ along the standard $t$-structure: $M^{<0}\to M\to H^0M\to M^{<0}[1]$. By the octahedral axiom we find a commutative diagram in the above right. Now we have $\Hom_{\per^\Z\!R}(M^{<0}[-1],Y[l])=0$, the morphism $Z'\to Y$ factors through $Z''$. Then the diagram $X\xleftarrow{s''}Z'\xrightarrow{}Y[l]$ with $\cone s''=H^0M\in\fl^\Z\!R$ gives the same morphism in $\qper R$ as the original one, which establishes our claim.
	
	Now let $X\to Y[l]$ be a morphism in $\qper R$ presented by the diagram $X\xleftarrow{s}Z\xrightarrow{}Y[l]$ in $\per R$ with $L:=\cone s\in\fl^\Z\!R$. Since $\Hom_{\per^\Z\!R}(L[-1],Y[l])=0$ by the assumption on $Y$, the map $Z\to Y[l]$ factors through $s$, hence we have $\Hom_{\qper R}(X,Y[l])\twoheadleftarrow\Hom_{\per^\Z\!R}(X,Y[l])=0$.
\end{proof}
Now we are ready to give our proof.
\begin{proof}[Proof of Proposition \ref{mm}]
	(\ref{tilt})  We first show the vanishing of extensions, that is, $\Hom_{\qper R}(T,T[i])=0$ for all $i\neq0$. The case $i<0$ follows from Lemma \ref{nex}. Also, when $i>d$ we have $\Hom_{\qper R}(T,T[i])=D\Hom_{\qper R}(T,T(a)[d-i])$ by Serre duality, thus $0$ again by Lemma \ref{nex}. Therefore it remains to consider the case $0\leq i\leq d$. Note that in $\per^\Z\!R$, we have $R(l)=(R(l)[-l])[l] \in \M[l]$, thus $T \in \M\ast\cdots\ast\M[a-1]$. Therefore, $T[i]$ lies in the fundamental domain $\M \ast \cdots\ast\M[d+a-1]$ for all $0\leq i \leq d$. This shows $\Hom_{\qper R}(T,T[i])=\Hom_{\per^\Z\!R}(T,T[i])=0$ for $0<i\leq d$. We next show that $T$ generates $\qper R$. Consider the minimal projective resolution $0 \to P_{d+1} \to \cdots \to P_0 \to R_0 \to 0$ of the graded $R$-module $R_0$. We know that $P_0=R$. Since $R$ is twisted $(d+1)$-CY of $a$-invariant $a$, we see that $P_{d+1}=R(a)$ and $P_i \in \add\{R(l)\mid 0<l<a \}$ for all $0<i<d+1$. (See the proofs of \cite[Proposition 4.3]{MM} or \cite[Lemma 3.8]{AIR}.) Therefore $R(a) \in \thick T$ in $\qper R$ and we see inductively that $T$ generates $\qper R$.\\
	(\ref{end})  We deduce by (\ref{tilt}) that there exists a triangle equivalence $\qper R \simeq \D^b(\md A)$. Comparing the Serre functor of each category, we have $(-)_\a(a)[d] \leftrightarrow \nu:=-\otimes^L_ADA$, thus $(-)_\a(a)\leftrightarrow \nu_d:=-\otimes^L_ADA[-d]$.
	
	We first show $\nu_d^{-i}A \in \md A$ for all $i\geq0$. For this we have to show $\Hom_{\D^b(\md A)}(A,\nu_d^{-i}A[l])=0$ for all $l\neq0$ and $i\geq0$. By the triangle equivalence, this is to show $\Hom_{\qper R}(T,T(-ia)[l])=0$ since $R_\a\simeq R$, or $\Hom_{\qper R}(T(ia),T[l])=0$. By Lemma \ref{nex} and the Serre duality we may assume $0<l\leq d$.
	To prove this we apply Theorem \ref{dact}(2).
	Note that $T(ia)\in \M\ast\cdots$ for all $i\geq0$ and $T[l]\in \cdots\ast\M[d+a-1]$ for all $l\leq d$. By Theorem \ref{dact}(2), we deduce $\Hom_{\qper R}(T(ia),T[l])=\Hom_{\per^\Z\!R}(T(ia),T[l])$, which is zero for $l\neq0$.
	
	Finally, we prove $\gd A\leq d$. Since $\gd A$ is certainly finite, it is sufficient to show $\Ext_A^l(DA,A)=0$ for $l>d$. For any $i>0$, we have $\Ext_A^{d+i}(DA,A)=\Hom_{\D^b(A)}(\nu A,A[d+i])=\Hom_{\D^b(A)}(A,\nu_d^{-1}A[i])$, which is $0$ by the previous claim.
\end{proof}

\section{On silting-cluster tilting correspondence}
Throughout this appendix, let $d$ be a positive integer, and let $\G$ be a bimodule $(d+1)$-CY DG algebra over a field $k$ such that $H^i\G=0$ for all $i>0$ and $H^0\G$ is finite dimensional. Amiot \cite{Am09} and Guo \cite{Guo} introduced the {\it cluster category} as 
\[ \C(\G):=\per\G/\D^b(\G). \]
Then $\G$ is a silting object in $\per\G$ and a $d$-cluster tilting object in $\C(\G)$ \cite[Theorem 2.1]{Am09}\cite[Theorem 2.2]{Guo}; see Theorem \ref{IYa}. More generally, the functor $\per\G\to\C(\G)$ sends each silting objet in $\per\G$ to a $d$-cluster tilting object in $\C(\G)$ \cite[Corollary 5.12]{IYa1}. Therefore we have a map
\begin{equation}\label{sct}
	\xymatrix{\silt\G\ar[r]& \ct{d}\G },
\end{equation}
where $\silt\G$ (resp. $\ct{d}\G$) is the set of isomorphism classes of silting objects in $\per\G$ (resp. $d$-cluster tilting objects in $\C(\G)$). The aim of this appendix is to discuss the following problem.
\begin{Q}
	How far is the map (\ref{sct}) being from bijective?
\end{Q}

A natural approach to study this question is to consider the full subcategory (called the {\it fundamental domain})
\[ \F=\mathcal{P}\ast\mathcal{P}[1]\ast\cdots\ast\mathcal{P}[d-1] \subset\per\G, \]
where $\mathcal{P}=\add\G\subset\per\G$. Then the composition $\F\subset\per\G\to\C(\G)$ is an additive equivalence (see Theorem \ref{IYa}). In particular the map (\ref{sct}) restricts to an injection
\begin{equation}\label{fsct}
	\xymatrix{\silt^\F\!\G\ar@{^(->}[r]& \ct{d}\G },
\end{equation}
where $\silt^\F\!\G$ is the subset of $\silt\G$ consisting of silting objects contained in $\F$.
\begin{Def}
	We call $\G$ {\it mild} if the map (\ref{fsct}) is bijective.
\end{Def}
For $d=1$ and $2$, $\G$ is always mild \cite{KN} (see also \cite[Corollary 5.12]{IYa1}). It was asked in \cite[Conjecture 5.14]{IYa1} if $\G$ is always mild for $d\geq3$. We will show that this is far from being true. In fact, under the assumption that $H^0\G=k$, we will characterize mildness as follows.
\begin{Thm}
	Let $\G$ be a $(d+1)$-Calabi-Yau DG $k$-algebra such that $H^i\G=0$ for $i>0$ and $H^0\G=k$. Then the following are equivalent.
	\begin{enumerate}
		\renewcommand{\labelenumi}{(\alph{enumi})}
		\renewcommand{\theenumi}{\alph{enumi}}
		\item\label{m} $\G$ is mild.
		\item\label{dco} $\ct{d}\G=\{\G,\ldots,\G[d-1]\}$.
		\item\label{qis} $\G$ is quasi-isomorphic to $k[x]$ with $\deg x=-d$ and with zero differentials.
		\item\label{tai} $\C(\G)$ is triangle equivalent to $\C_d(k)$, the $d$-cluster category of $k$.
	\end{enumerate}
\end{Thm}
\begin{proof}
	(\ref{m})$\Leftrightarrow$(\ref{dco})  Since $H^0\G=k$, $\G$ is an indecomposable silting object in $\per\G$, thus $\silt\G=\{\G[i]\mid i\in\Z\}$. Then $\silt^\F\G=\{\G[i]\mid0\leq i\leq d-1\}$ and the desired equivalence follows.\\
	(\ref{dco})$\Rightarrow$(\ref{qis})  We first show that $\G$ is $d$-periodic in $\C(\G)$, that is, $\G\simeq\G[d]$ in $\C(\G)$. Suppose that $\ct{d}\G=\{\G[i]\mid0\leq i\leq d-1\}$. Then the $d$-cluster tilting object $\G[d]\in\C(\G)$ has to be isomorphic to $\G[i]$ for some $0\leq i\leq d-1$. The only possible $i$ is $0$ since $\Hom_{\C(\G)}(\G,\G[i])=0$ for $1\leq i\leq d-1$. Thus we have $\G\simeq\G[d]$ in $\C(\G)$.\\
	We next compute the cohomology of $\G$. By Theorem \ref{IYa}(2) the functor $\per\G\to\C(\G)$ induces bijections $\Hom_{\per\G}(\G[i],\G)\to\Hom_{\C(\G)}(\G[i],\G)$ for each $i\geq0$, thus $H^{-i}\G=\Hom_{\C(\G)}(\G[i],\G)$ for $i\geq0$. By periodicity of $\G$, this is $k$ if $i\mid d$, and $0$ if $i\mathrel{\not|}d$ since $\Hom_{\C(\G)}(\G,\G[i])=0$ for $0<i<d$. Now lift an isomorphism $\G[d]\to\G$ in $\C(\G)$ to a morphism $f\colon\G[d]\to\G$ in $\per\G$, and let $y\in Z^{-d}\G$ give the morphism $f$ in $H^{-d}\G$. Then we obtain a homomorphism $k[x]\to\G$ of DG algebras, taking $x$ to $y$. Consider the power $y^n\in Z^{nd}\G$ of $y$. It presents the morphism $\G[nd]\xrightarrow{f[(n-1)d]}\G[(n-1)d]\to\cdots\xrightarrow{f[1]}\G[1]\xrightarrow{f}\G$ in $\per\G$ which is an isomorphism in $\C(\G)$. Therefore $y^n$ is non-zero in $H^{-nd}\G$. We conclude that $k[x]\to\G$ is a quasi-isomorphism.\\
	(\ref{qis})$\Rightarrow$(\ref{tai})  Since $k[x]$ with $\deg x=-d$ is the derived $(d+1)$-preprojective algebra of $k$, the assertion follows.\\
	(\ref{tai})$\Rightarrow$(\ref{dco})  If $\C(\G)\simeq\C_d(k)$ then clearly $\ct{d}\G=\{\G[i]\mid0\leq i\leq d-1\}$.
\end{proof}

\begin{Ex}
	We regard the polynomial ring $\G=k[x_1,\ldots,x_n]$ as a DG algebra with $\deg x_i=-a_i<0$ and zero differential. We assume that $a=\sum_{i=1}^na_i$ is odd. Then $\G$ is an $(n+a)$-CY DG algebra by Theorem \ref{dgcy}. By (\ref{qis}) of the above theorem we immdediately deduce the following.
	\begin{Cor}
		$\G$ is mild if and only if $n=1$.
	\end{Cor}
\end{Ex}
\end{appendix}

\end{document}